\documentclass[reqno]{amsart}
\makeatletter
\@namedef{subjclassname@2020}{%
  \textup{2020} Mathematics Subject Classification}
\makeatother
\usepackage{amssymb}
\usepackage[utf8]{inputenc} 
\usepackage[T1]{fontenc}    
\usepackage{url}            
\usepackage{booktabs}       
\usepackage{amsfonts}       
\usepackage{nicefrac}       
\usepackage{microtype}      
\usepackage{color}
\usepackage{lipsum}
\usepackage{float}
\usepackage{calrsfs}
\usepackage{enumerate}
\usepackage{tikz-cd}
\usepackage{tikz}
\usepackage{graphicx}
\usepackage{old-arrows} 
\usetikzlibrary{arrows}
\usepackage{comment}
\DeclareMathAlphabet{\pazocal}{OMS}{zplm}{m}{n}

\usepackage{eqnarray,amsmath}
\usepackage{marginnote}
\usepackage{mathtools}
\usepackage{amsbsy}
\usepackage{hyperref}

\input xygraph
\input xy
\xyoption{all}

\usepackage{makecell,multirow}
\usepackage{rotating}
\newcolumntype{C}[1]{>{\centering\let\newline\\\arraybackslash\hspace{0pt}}m{#1}}

\def\alg{\mathop{\text{\bf alg}}}
\def\grp{\mathop{\text{\bf Grp}}}
\def\set{\mathop{\text{\bf Set}}}
\def\Gm{\mathop{\text{\bf G}_{\text{\bf m}}}}
\def\Ga{\mathop{\text{\bf G}_{\text{\bf a}}}}
\def\Mu{\mathop{\boldsymbol{\mu}}}
\def\F{\mathcal F}
\def\H{\mathcal H}
\def\U{\mathcal U}
\def\T{\mathcal T}
\DeclareMathOperator{\aut}{Aut}
\def\affaut{\mathop{\text{\bf Aut}}}
\def\o{\otimes}
\def\a{\alpha}
\def\b{\beta}
\def\l{\lambda}
\def\m{\mu}
\def\s{\sigma}
\def\w{\omega}

\def\affaut{\mathop{\text{\bf aut}}}
\def\E{A}
\def\remove#1{}
\def\rep{\mathop{\hbox{\bf rep}}}
\def\mi{\mathfrak{m}}
\DeclareMathOperator{\car}{char}

\newtheorem{lemma}{Lemma}[section]
\newtheorem{corollary}[lemma]{Corollary}
\newtheorem{theorem}[lemma]{Theorem}
\newtheorem{proposition}[lemma]{Proposition}
\newtheorem{remark}[lemma]{Remark}
\newtheorem{definition}[lemma]{Definition}

\usepackage{arydshln}

\title[Hopf algebras in evolution algebras]{
Hopf algebras and associative representations of two-dimensional evolution algebras.
}
\author[Y. Cabrera]{Yolanda Cabrera Casado}
\address{Y. Cabrera Casado: Departamento de Matem\'atica Aplicada, E.T.S. Ingenier\'\i a Inform\'atica, Universidad de M\'alaga, Campus de Teatinos s/n. 29071 M\'alaga.   Spain.}
\email{yolandacc@uma.es}

\author[M. I. Gon\c calves]{Maria Inez Cardoso Gon\c calves}
\address{M. I. Cardoso Gon\c calves: Departamento de Matem\'atica, Universidade Federal de Santa Catarina, Florian\'opolis, SC, 88040-900 - Brazil}
\email{maria.inez@ufsc.br}

\author[D. Gon\c calves]{Daniel Gon\c calves}
\address{D. Gon\c calves: Departamento de Matem\'atica, Universidade Federal de Santa Catarina, Florian\'opolis, SC, 88040-900 - Brazil}
\email{daemig@gmail.com}

\author[D. Mart\'\i n]{Dolores Mart\'\i n Barquero}
\address{D. Mart\'\i n Barquero: Departamento de Matem\'atica Aplicada, Escuela de Ingenier\'\i as Industriales, Universidad de M\'alaga, Campus de Teatinos s/n. 29071 M\'alaga.   Spain.}
\email{dmartin@uma.es}

\author[C. Mart\'\i n]{C\'andido Mart\'\i n Gonz\'alez}
\address{C. Mart\'\i n Gonz\'alez: Departamento de \'Algebra Geometr\'{\i}a y Topolog\'{\i}a, Fa\-cultad de Ciencias, Universidad de M\'alaga, Campus de Teatinos s/n. 29071 M\'alaga.   Spain.}
\email{candido\_m@uma.es}

\author[I. Ruiz]{Iv\'an Ruiz Campos}
\address{I. Ruiz Campos:  Departamento de \'Algebra Geometr\'{\i}a y Topolog\'{\i}a, Fa\-cultad de Ciencias, Universidad de M\'alaga, Campus de Teatinos s/n. 29071 M\'alaga. Spain.}
\email{ivaruicam@uma.es}
\thanks{The first,  fourth, fifth, and sixth authors are supported by the Spanish Ministerio de Ciencia e Innovaci\'on   through project  PID2019-104236GB-I00/AEI/10.13039/501100011033 and PID2023-152673NB-I00
 and by the Junta de Andaluc\'{i}a  through projects  FQM-336 and UMA18-FEDERJA-119,    all of them with FEDER funds.  The third author was partially supported by Capes-Print Brazil, Conselho Nacional de Desenvolvimento Cient\'ifico e Tecnol\'ogico (CNPq) - Brazil, and Funda\c{c}\~ao de Amparo \`a Pesquisa e Inova\c{c}\~ao do Estado de Santa Catarina (FAPESC). The fourth and fifth authors were supported by Brazilian Federal Agency for Support and Evaluation of Graduate Education â Capes with Process numbers: 88887.895676/2023-00 and 88887.895611/2023-00 respectively. The sixth author is supported by a Junta de Andalucía PID fellowship no. PREDOC\_00029. The second and third authors are partially supported  by the Junta de Andaluc\'{i}a project FQM-336.} 


\begin{document}
\maketitle
\begin{abstract}
In this paper, we establish a connection between evolution algebras of dimension two and Hopf algebras, via the algebraic group of automorphisms of an evolution algebra.   
    Initially, we describe the Hopf algebra associated with the automorphism group of a 2-dimensional evolution algebra. Subsequently, for a 2-dimensional evolution algebra $A$ over a field $K$, we detail the relation between the algebra associated with the (tight) universal associative and commutative representation of $A$, referred to as the (tight) $p$-algebra, and the corresponding Hopf algebra, $\mathcal{H}$, representing the affine group scheme $\text{Aut}(A)$. 
    Our analysis involves the computation of the (tight) $p-$algebra associated with any 2-dimensional evolution algebra, whenever it exists. We find that $\text{Aut}(A)=1$  if and only if there is no faithful associative and commutative representation for $A$. Moreover, there is a faithful associative and commutative representation for $A$ if and only if $\mathcal{H}\not\cong K$ and $\text{char} (K)\neq 2$, or $\mathcal{H}\not\cong K(\epsilon)$  (the dual numbers algebra) and $\mathcal{H}\not\cong K$ in case of $\text{char} (K)= 2$. Furthermore, if $A$ is perfect and has a faithful tight $p$-algebra, then this $p$-algebra is isomorphic to $\mathcal{H}$ (as algebras). Finally, we derive implications for arbitrary finite-dimensional evolution algebras.
\end{abstract}

\section{Introduction}

Evolution algebras, introduced as models for non-Mendelian genetics in Tian's 2004 Ph.D. Thesis \cite{tian}, have since sparked extensive research efforts, as evidenced in \cite{ceballos}. In another vein, Hopf algebras and their relations with affine group schemes have become an important topic of study in the the field of Algebra, garnering considerable attention from researchers, \cite{Waterhouse}. In this work, we establish 
a connection between these two fundamental concepts via the algebraic group of automorphisms of a two-dimensional evolution algebra.

Central to our study is the notion of associative representation for nonassociative algebras, as introduced by Shestakov and Kornev in \cite{shestakov}. Notably, they found an associative (faithful) representation of octonions within an $8\times 8$-matrix algebra. 
Before we go into the technical details of this concept, let us write a few lines on the philosophy that inspires it. 
Consider the algebra
$$\U=K[x,y]/(x^4-x,y^4-y,xy)$$ with involution $*\colon \U\to\U$ induced by  $\bar x^*=\bar y^2$ and $\bar y^*=\bar x^2$ (where the bar denotes class modulo the ideal $(x^4-x,y^4-y,xy)$). If we define in $\U$ the multiplication $p\colon \U\times\U\to\U$ given by $p(a,b)=a^*b^*$, then we have the following: $p(\bar x,\bar y)=\bar x^*\bar y^*=\bar y^2\bar x^2=0$, $p(\bar x,\bar x)=(\bar x^*)^2=\bar y^4=\bar y$ and $p(\bar y,\bar y)=(\bar y^*)^2=\bar x^4=\bar x$. In the sequel, we will prove  that $\{\bar x,\bar y\}$ is a linearly independent subset of $\U$, and the multiplication of the algebra $A:=K\bar x\oplus K\bar y$ relative to the product $p$ is represented by the following table:
\begin{center}
\begin{tabular}{|c||c|c|}
\hline 
  $p$  & $\bar x$ & $\bar y$\\
  \hline 
  $\bar x$ & $\bar y$ & $0$\\
  $\bar y$ & $0$ & $\bar x$\\
\hline     
\end{tabular}
\end{center}
Thus, the $2$-dimensional evolution algebra $A$ is constructed as a subalgebra of a polynomial algebra relative to the product $p$. 
It is important to note that while $A$ is not associative, it can be embedded into an associative algebra with involution (though we must pay the price of replacing the associative product of $\U$ with the product $p$).
As we will see, $\U$ is the universal associative and commutative representation of $A$. 

We now give a more detailed description of the construction of the universal associative and commutative representation of a 2-dimensional evolution algebra. 
Following up the work \cite{squares}, given a field $K$, we consider a two-dimensional evolution $K$-algebra $A$ with natural basis $\{e_i\}_{i=1}^2$ such that $e_i^2=\sum_{j=1}^2\w_{ji}e_j$, $w_{ji}\in K$, the polynomial algebra  in four variables $K[x,y,x^*,y^*]$, a formal polynomial $p(a,b)=\lambda_0 ab + \lambda_1 ab^*+ \lambda_2 a^*b+ \lambda_3 a^*b^*$, and the ideal $I\triangleleft K[x,y,x^*,y^*]$ generated by the elements $p(x,x)-\w_{11} x-\w_{12} y$, $p(y,y)-\w_{21} x-\w_{22} y$, $p(x,y)$, $p(y,x)$ and $p(x^*,x^*)-\w_{11} x^*-\w_{12} y^*$, $p(y^*,y^*)-\w_{21} x^*-\w_{22} y^*$, $p(x^*,y^*)$, $p(y^*,x^*)$.
Then, we make the quotient $\U_p:=K[x,y,x^*,y^*]/I$ equipped with an involution induced by the map $\bar x\mapsto \overline{x^*}$ and similarly for $\bar y$. 
Finally, we consider the product $\U_p\times\U_p\to \U_p$ such that $(z,t)\mapsto \lambda_0 zt + \lambda_1 zt^*+ \lambda_2 z^*t+ \lambda_3 z^*t^*$.
This construction aims to find an algebra homomorphism $\rho \colon A \to (\U_p,p)$ such that $(\U_p,p)$ is the universal (associative and commutative) representation of $A$, in the sense that if we have another representation $\sigma \colon A \to V$, with $V$ an associative and commutative unital $K$-algebra with involution, then there exists a $*$-homomorphism of unital associative algebras with involution $\theta: \U_p \to V$ such that $\sigma =  \theta \circ \rho $. Similarly, we also use a kind of tight universal associative and commutative representation (not necessarily unital) $\T_p$ related to $\U_p$, which is, in a certain sense, \lq\lq smaller\rq\rq\ than $\U_p$. 
The universal homomorphism $\rho\colon A\to\U_p$ factorizes through $\T_p$, so $\T_p$ also has a suitable universal property.

Given the seemingly untamed nature of evolution algebras, attributed to their modest definitional requirements, it is interesting to inquire about those evolution algebras that can be embedded in associative algebras somehow.  
One of our goals is to identify which $2$-dimensional evolution algebras admit a faithful representation in a polynomial algebra. 

In \cite{squares}, we conjectured that when the group of automorphisms of a two-dimensional evolution algebra is trivial, there is no nontrivial associative representation $\U_p$. 
However, we lacked a definitive proof of this fact, due to inconclusive computer calculations (refer to \cite[Remark 6.9]{squares}).
In this work, we present a proof of this conjecture.

As we mentioned before, one of our key results is to establish a connection between the universal associative and commutative representation of an evolution algebra $A$ and the Hopf algebra representing its automorphism group (as an affine group scheme). Specifically, for any $2$-dimensional evolution algebra $A$, we compute the Hopf algebra $\H$ such that $\affaut(A)\cong\hom_{\alg_K}(\H,\_)$ and compare it with $\U_p$ for specific $p$.
In the perfect case, we prove that $\H\cong\T_p$ as algebras, prompting conjectures inspired by the $2$-dimensional scenario. Finally, we obtain some consequences that extend beyond the $2$-dimensional setting.

The paper is organized as follows. In Section~\ref{Ivanov}, we set up notation, recall the classification of evolution algebras of dimension two, and describe the Hopf algebras representing the affine group scheme of automorphisms of the two-dimensional evolution algebras. Next, in Section~\ref{morrocotudo}, we describe associative and commutative representations of the algebras studied in Section~\ref{Ivanov}. We finish the paper in Section~\ref{banco}, where we provide the notion of associative and commutative representation of an evolution algebra (a minor variation on the notion of associative and commutative representation of \cite{shestakov}) and present results connecting the associative and commutative representation of a finite-dimensional evolution algebra with its automorphism group and corresponding Hopf algebra.

\section{Affine group scheme of automorphisms}\label{Ivanov}

Let $K$ be a field. We denote the category of associative, commutative, and unital $K$-algebras by $\alg_K$. If $A, B$ are objects of this category; we define $\hom_{\alg_K}(A, B)$ as the set of all the homomorphisms of unital $K$-algebras from $A$ to $B$.
If $R\in\alg_K$ and $M,M',N,N'$ are $R$-modules then, given $R$-module homomorphisms $T\colon M\to N$ and $T'\colon M'\to N'$, we denote by $T\o T'\colon M\o M'\to N\o N'$ the $R$-module homomorphism such that
$(T\o T')(m\o m')=T(m)\o T'(m')$, for any $m\in M$ and $m'\in M'$.
The category of groups is denoted by $\grp$ and the category of sets by $\set$. A functor $\F\colon\alg_K\to \set$  is referred to as a $K$-set functor, while a functor $\F\colon\alg_K\to\grp$ is called a $K$-group functor. 

If $U$ is any $K$-algebra, whether associative or not, one can consider its scalar extension $U_R:=U\otimes R$ for any $R\in\alg_K$. Recall that  $U_R$ is an $R$-algebra with product given by $(u\o r)(u'\o r')=uu'\o rr'$ for any $u,u'\in U$ and $r,r'\in R$. We can legitimately consider the group $\aut_R(U_R)$ (for any $R\in\alg_K$). This allows us to define a $K$-group functor $\affaut(U)\colon\alg_K\to\grp$ such that 
$\affaut(U)(R)=\aut_R(U_R)$ and, for $\a\in\hom_{\alg_K}(R,R')$, we define 
$\affaut(U)(\alpha)\colon \aut_R(U_R)\to \aut_{R'}(U_{R'})$ such that any $f\in\aut_R(U_R)$ maps to $g\in \aut_{R'}(U_{R'})$ defined as $g(u\otimes 1_{R'}):=(1\otimes\a)f(u\otimes 1_R)$ where $1_R$ is the unit of $R$ and similarly for $1_{R'}.$
 Other functors that we use are:
\begin{enumerate}[\rm (i)]
\item The \emph{multiplicative group} functor $\Gm\colon\alg_K\to \grp$, such that $\Gm(R)=R^\times.$ 
\item The \emph{$n$\textsuperscript{th} roots of the unity group} functor ${\boldmath\Mu_n} \colon\alg_K\to\grp$, such that $\Mu_n(R)=\{r\in R\colon r^n=1\}$.
\item The \emph{additive group} functor $\Ga\colon\alg_K\to\grp$, such that $\Ga(R)=(R,+)$, the underlying additive group of $R$. 
\end{enumerate}

In \cite{squares}, evolution algebras of dimension $2$ over arbitrary fields are classified into $9$ cases. The automorphisms and derivations of these algebras are classified in the same reference, to which we refer the reader for notations. The cases of the classification in \cite{squares} are (ruling out the zero product algebra $A_0$): $A_1$, $A_{2,\alpha}$, $A_{3,\alpha}$, $A_{4,\alpha}$, $A_{5,\alpha,\b}$, $A_5$, $A_6$, $A_7$ and $A_{8,\alpha}$. In this section, we describe the Hopf algebras representing the affine group scheme of automorphisms of these algebras. 
 For the perfect algebras in the above classification, the Hopf algebra representing $\affaut(A)$ can be obtained by applying the results from \cite{Elduque}. However, to include the non-perfect algebras as well, in some cases, we have chosen a direct approach.

\subsection{The $\text{A}_\text{1}$ algebra} \label{subsec_A1}
This is $A_1=Ke_1\oplus Ke_2$ with natural basis $\{e_1,e_2\}$ satisfying $e_i^2=e_i$ for $i=1,2$. 
For any $R\in\alg_K$, the scalar extension $(A_1)_R$ is identified with the $R$-algebra $(A_1)_R=Re_1\oplus Re_2$ with multiplication $e_i^2=e_i$ for $i=1,2$ (and of course $e_1e_2=e_2e_1=0$). 
Following the results presented by A. Elduque and A. Labra in \cite{Elduque}, one gets that the representing Hopf algebra of $\affaut(K^N)$,
where $K^N$ is the evolution algebra with component-wise product, $N\ge 1$, is $K^{N!}$ itself. In our case, $N=2$ with the standard Hopf algebra structure given in \cite[2.3, p. 16]{Waterhouse}. 


\subsection{The $\boldsymbol{\text{A}_{2,\alpha}}$ algebra}
For this algebra, the natural basis $\{e_1,e_2\}$ satisfies $e_1^2=e_2$ and $e_2^2=\alpha e_1$, with $\alpha \in K^\times$. The results of \cite[Corollary 3.4]{Elduque} give us a short exact sequence of group schemes $1\to \Mu_3\to\affaut(A_{2,\alpha})\to \boldsymbol{H}\to 1$ where $\boldsymbol{H}$ is the constant group scheme of $C_2$ (in case $\alpha=1$). However, in this case, we will determine the group $\affaut(A_{2,\alpha})$ in a direct way. If we consider $f\in\aut_R((A_{2,\alpha})_R)$, then there are elements $a,b,c,d\in R$ such that $f(e_1)=ae_1+be_2$ and $f(e_2)=ce_1+de_2$.

 To be an invertible homomorphism of algebras, we need that $f(e_1^2)=f(e_1)^2$, $f(e_2^2)= f(e_2)^2$, $f(e_1)f(e_2)=0$, and $ad-bc\in R^\times$. This implies that the parameters $a,b,c$ and $d$ must satisfy $c=\alpha b^2, a^2 =d,
a=d^2, c^2=\alpha b, ac=0$ and $ \alpha bd=0$.
 
Next, we can eliminate $d$ and $c$ to get 
\begin{equation}\label{petra}
 a=a^4, \,  b=\alpha b^4, \,
ab^2=ba^2=0, \, a^3-\alpha b^3\in R^\times.\end{equation}
Of course, we must remember that $d=a^2$ and $c=\alpha b^2$ are used to recover the matrix of $f$. 

Now, let $z:=a^3-\alpha b^3\in R^\times$. Then $z^2=a^3+\alpha b^3$ and hence $z^3=z$. Whence $z^2=1$, that is, $a^3+\alpha b^3=1$. So, we notice that the system in \eqref{petra} is equivalent to 
\begin{equation}\label{skol}
\begin{cases} 
ab=0, \\ a^3+\alpha b^3=1.\end{cases} \end{equation}
Any solution $(a,b)$ of \eqref{skol} is a solution of \eqref{petra} (since 
$(a^3-\alpha b^3)^2=a^6+\alpha^2 b^6=a^3+\alpha b^3=1$) which implies that $a^3-\alpha b^3\in R^\times$. Conversely, any pair $(a,b)$ satisfying \eqref{petra} also satisfies \eqref{skol} because $a^3-\alpha b^3\in R^\times$ gives $a^3+\alpha b^3=1$, as we have seen above. Then $ab=ab(a^3+\alpha b^3)=0$. Therefore, 
$$\text{Aut}_R((A_{2,\alpha})_R)\cong\left\{
\begin{pmatrix}a & b\\ \alpha b^2 & a^2\end{pmatrix}\colon a,b \in R, ab=0, a^3+\alpha b^3=1\right\}.$$

The representing Hopf algebra of the group scheme $\affaut(A_{2,\alpha})$ is given by
$\mathcal{H}_{2,\alpha} := K[a,b]/I$, where $I$ is the ideal $I=(ab, a^3+\alpha b^3-1)$. Consider the functor $G_{2,\alpha}\colon \alg_K \to \grp$ with $G_{2,\alpha}(R) := \{(a,b) \in R^2 \colon ab = 0, a^3+\alpha b^3 = 1\}$, where the product is given by $(a,b) * (c,d)  = (ac+\alpha bd^2,ad+bc^2)$. Multiplication induces the morphism $\Delta \colon \mathcal{H}_{2,\alpha} \to \mathcal{H}_{2,\alpha} \otimes \mathcal{H}_{2,\alpha}$ given by $\Delta(\bar a) = \bar a \otimes \bar a + \alpha \bar b \otimes \bar b^2 $ and $\Delta(\bar b) = \bar a \otimes \bar b + \bar b \otimes \bar a^2$. Thanks to matrix inverses, we can compute the inverse of an element $(a,b)$, which is given by $(a^2,b)$. This induces $S \colon \mathcal{H}_{2,\alpha} \to \mathcal{H}_{2,\alpha}$ with $S(\bar a) = \bar a^2 $ and $S(\bar b) = \bar b$. Finally, the augmentation map induced by the neutral element is $\varepsilon \colon \mathcal{H}_{2,\alpha} \to K$, with $\varepsilon(\bar a) = 1$ and $\varepsilon (\bar b) = 0$. It is not hard to prove that $\dim(\H_{2,\alpha})=6$.

So, a basis of $\H_{2,\alpha}$ is the set $\{\bar a,\bar a^2,\bar a^3,\bar b,\bar b^2,\bar b^3\}$. Furthermore, since $\bar a\bar b=0$ we have two orthogonal ideals $H_1:=K\bar a\oplus K\bar a^2\oplus K\bar a^3$ and $H_2:=K\bar b\oplus K\bar b^2\oplus K\bar b^3$. Notice that $\bar a^3$ and $\alpha \bar b^3$ are orthogonal idempotents whose sum is $1$.
In the particular case that $\root 3 \of\alpha\in K$, we can go a little further describing $\H_{2,\alpha}$. Indeed, in this case, it is not hard  to realize that $\H_{2,\alpha}\cong\H_{2,1}$, so that we can take $\alpha=1$. 
Then, each ideal $H_j$ with $j\in \{1,2\}$ is isomorphic to the group algebra of the cyclic group of three elements $C_3$. Thus, $\H_{2,1}\cong KC_3\oplus KC_3$ as $K$-algebras, where the first direct summand can be identified with $KC_3=K 1\oplus Ka\oplus Ka^2$ and the second $KC_3=K 1\oplus Kb\oplus Kb^2$. Notice that an involution of $\H_{2,1}$, module the above identification, is given by $a^*=b^2$ and $b^*=a^2$. Finally, an isomorphism $KC_3 \oplus KC_3 \cong KC_3 \otimes K^2$ is given by $(a^i,0)\mapsto a^i\o e_1$, $(0,a^i)\mapsto a^i\o e_2$, where $\{e_1,e_2\}$ is the canonical basis of $K^2$.


\subsection{The $\boldsymbol{\text{A}_{3,\alpha}}$  and $\boldsymbol{\text{A}_{4,\alpha}}$  algebras }
Fixing the natural basis $\{e_1,e_2\}$, we have that: $\text{A}_\text{3,$\alpha$}$  is given by  $e_1^2= e_1$; $e_2^2= \alpha e_1 + e_2$ and $\text{A}_\text{4,$\alpha$}$ is such that $e_1^2= \alpha e_2$, $e_2^2= e_1 + e_2$,  where $\alpha \in R^\times$. Applying \cite[Corollary 3.4]{Elduque}, as well as the computation of \rm {Diag}$(\Gamma)$ as in \cite[Section 2]{Elduque}, we get $\aut_R((A_{3,\alpha})_R)= \aut_R((A_{4,\alpha})_R)=\tiny\left\{ {\begin{pmatrix}
        1 & 0 \\
        0 & 1
    \end{pmatrix}} \right\} $. A direct approach to getting this result is also feasible by posing the equations of a generic automorphism and solving them. 


\subsection{The $\boldsymbol{\text{A}_\text{5,$\alpha$,$\b$}}$ algebra}

In this case, the natural basis $\{e_1,e_2\}$ is such that $e_1^2=e_1+\b e_2$, $e_2^2=\alpha e_1+e_2$ with $\alpha,\b\in K^\times$ and $\alpha\b\ne 1$.
Following once again \cite{Elduque}, we obtain that the affine group scheme {\rm Diag}($\Gamma$)=1, and thus, we have the exactness (see \cite[Formula (8)]{Elduque}) of the sequence $1\to 1\to \affaut(A_{5,\alpha,\b})\to \boldsymbol{H}\to 1$, where $\boldsymbol{H}$ is the constant group scheme associated to a certain 
subgroup $H$ of $\aut(\Gamma)$. In case $\alpha \neq \beta $ it turns out that $\boldsymbol{H}=1$, because $\aut(\Gamma)=1$. Consequently, the exact sequence is $1\to \affaut(A_{5,\alpha,\b})\to 1$ implying $\affaut(A_{5,\alpha,\b})=1$. Once again, from \cite{Elduque} one has that
the automorphism group  $\aut_R \left ((A_{5,\alpha,\b})_R\right)=\{1\}$ for any $R\in\alg_K$. 
In case $\alpha=\beta$, the group $\aut(\Gamma)$ is the cyclic group of order $2$, and one can check that we have a short exact sequence 
$1\to \affaut(A_{5,\alpha,\alpha})\to \boldsymbol{C}_2\to 1$ which gives $\affaut(A_{5,\a,\a})\cong\boldsymbol{C}_2$ and the representing Hopf algebra is $K^2$ with componentwise product and standard Hopf algebra structure, as in Subsection \ref{subsec_A1}.

\subsection{The $\text{A}_\text{5}$ algebra}
This is the two-dimensional algebra generated by the natural basis $\{e_1,e_2\}$, which verifies $e_1^2 = e_1 -e_2$ and $e_2^2 = e_2-e_1$. If the linear application $f \colon (A_5)_R \to (A_5)_R$ is defined by $f(e_1) = ae_1 + be_2$ and $f(e_2) = ce_1+de_2$ then, for $f$ to be an automorphism, the next equations have to be satisfied:
\begin{equation*}
       \left \{ 
\begin{array}{rl}
        (a-c)e_1+ (b-d)e_2 & = (a^2-b^2)e_1+(b^2-a^2)e_2, \\
         (c-a)e_1+(d-b)e_2  &= (c^2-d^2)e_1 + (d^2-c^2)e_2, \\
         0 & = (ac-bd)e_1+ (bd-ac)e_2, \\
         (ad-bc)u  & = 1. 
\end{array} 
\right.
\end{equation*}
This is equivalent to $      a-c =  a^2-b^2$, $ b-d =  b^2 -a^2$, $  c-a = c^2-d^2$, $d-b  = d^2 -c^2$, $0  = ac-bd$, $ (ad-bc)u   = 1. $

From these equations, and considering that we are in a unitary commutative ring (not necessarily a domain), we have that $b = c = 1-a$, $d = a$, and $2a -1 \in R^\times$. Next, we identify two cases depending on the characteristic of $K$.

\begin{enumerate}[\rm (i)]
    \item If $\text{char}(K) = 2$ we have
        $\aut\nolimits_R((A_5)_R) \cong \left \{ \footnotesize{\begin{pmatrix}
        a & 1-a \\
        1-a & a
    \end{pmatrix}} \colon a \in R \right \} $.
The representing algebra of $\affaut(A_5)$ is $\mathcal{H}_5 = K[x]$ and  $\affaut(A_5)$ is isomorphic to the affine group scheme $\mathcal{G}_{5} \colon \alg_K \to \grp$ given by $\mathcal{G}_5(R) = R$ with the product $a*b = 1+a+b$. The neutral element is $1$ and the inverse of $a$ is $a$. These three conditions induce the following morphisms. The application $\Delta \colon \mathcal{H}_{5} \to \mathcal{H}_{5}\otimes \mathcal{H}_{5}$ with $\Delta(x) = x \otimes 1 + 1 \otimes x + 1\otimes 1$, the morphism $S \colon \mathcal{H}_{5} \to \mathcal{H}_{5}$ with $S(x) = x$ and the augmentation $\varepsilon \colon \mathcal{H}_{5} \to K$ with $\varepsilon(x) = 1$. 
    Alternatively, we can give a group isomorphism 
     $\left\{\footnotesize{\begin{pmatrix}
        a & 1+a \\
        1+a & a
    \end{pmatrix}} \colon a \in R \right\}{\buildrel{\omega}\over\to} (R,+),$
    such that $\omega \left (\footnotesize{\begin{pmatrix}
        a & 1+a \\
        1+a & a
    \end{pmatrix}} \right)=1+a$. So, we have that $\aut_R((A_5)_R)\cong (R,+)$ and the representing
    Hopf algebra of this affine group scheme is $K[x]$. Of course, we have that $\affaut(A_5)\cong\Ga$. 
    \item If $\text{char}(K) \neq 2$ then
       $$\aut\nolimits_R((A_5)_R) \cong \left \{ {\footnotesize\begin{pmatrix}
        a & 1-a \\
        1-a & a
    \end{pmatrix}} \colon a \in R, 2a-1 \in R^\times  \right \}.$$
    In this case, $\mathcal{H}_5=K[x,y]/(2xy-y-1)$ is the representing algebra of $\affaut(A_5)$.  We consider the affine group scheme $\mathcal{G}_5\colon \alg_K \to \grp$ with $\mathcal{G}_5(R) = \{(a,b) \in R^2 \colon b(2a-1) = 1\}$. The product is given by $(a,b)*(c,d) = (1-a-c+2ac, bd)$. The neutral element is $(1,1)$ and the inverse of $(a,b)$ is $(ab,2a-1)$. This leads us to $\Delta \colon \mathcal{H}_5 \to \mathcal{H}_5\otimes \mathcal{H}_5$ with $\Delta(\bar x) = 1\otimes 1 - \bar x\otimes 1 - 1 \otimes \bar x + 2 \bar x \otimes \bar x$ and $\Delta (\bar y) = \bar y \otimes \bar y$. The coinverse morphism is given by $S\colon \mathcal{H}_5 \to \mathcal{H}_5$ with $S(\bar x) = \bar x \bar y$ and $S(\bar y) = 2\bar x -1$. Finally, the augmentation map is $\varepsilon\colon \mathcal{H}_5 \to K$ with $\varepsilon(\bar x) = 1$ and $\varepsilon(\bar y) = 1$. Alternatively, note that the determinant map provides an isomorphism $\aut_R((A_5)_R){\buildrel{\det}\over\to} R^\times$. So, we can say that $\aut_R((A_5)_R)\cong R^\times$. Hence, $\H_5\cong K[x^\pm]$ (the Laurent polynomials algebra) and $\affaut(A_5)\cong\Gm$. 
\end{enumerate}

\subsection{The $\text{A}_\text{6}$ algebra.}
For this algebra, the natural basis $\{e_1,e_2\}$ is such that $e_1^2 = 0$, $e_2^2 = e_1$. Let us consider the linear application $f \colon (A_6)_R \to (A_6)_R$ given by $f(e_1) = ae_1+be_2$ and $f(e_2) = ce_1+de_2$ with $a,b,c,d \in R$. This application needs to verify the following equations to be an automorphism of algebras: $0 = b^2 e_1, ae_1+be_2= d^2e_1, 0 = bd e_1, (ad-bc)u = 1 $
with $u \in R$. These equations are equivalent to $b =0$, $a = d^2$ and $d^3u = 1$. So, we have that $\aut\nolimits_R((A_6)_R) \cong \left \{ {\footnotesize\begin{pmatrix}
        d^2 & 0 \\
        c & d
    \end{pmatrix}} \colon c \in R, d \in R^\times \right \}$.
Notice that given two elements ${\footnotesize\begin{pmatrix}
        d^2 & 0 \\
        c & d
    \end{pmatrix}}$ and
    ${\footnotesize\begin{pmatrix}
        s^2 & 0 \\
        r & s
    \end{pmatrix}}$ in $\aut\nolimits_R((A_6)_R)$, their multiplication is ${\footnotesize\begin{pmatrix}
        d^2s^2 & 0 \\
        cs^2+dr & ds
    \end{pmatrix}}$. This tells us that $\aut\nolimits_R((A_6)_R) \cong (R \times R^\times,*)$ with $(c,d)*(r,s) := (cs^2+dr,ds)$.  The corresponding representing algebra is $\mathcal{H}_6 =K[x,y^{\pm}]$. The neutral element is $(0,1)$ and $(c,d)^{-1} = (-cd^{-3},d^{-1})$. These induce $\Delta \colon\mathcal{H}_6 \to \mathcal{H}_6\otimes \mathcal{H}_6$ given by $\Delta (x) = x\otimes y^2+y\otimes x $ and $\Delta(y) = y\otimes y$, $S \colon \mathcal{H}_6 \to\mathcal{H}_6$ given by $S(x) = -xy^{-3}$ and $S(y) = y^{-1}$, and the augmentation $\varepsilon \colon \mathcal{H}_6 \to K$ given by $\varepsilon(x) = 0$ and $\varepsilon(y) = 1$. 

 Observe that any matrix as in $\aut\nolimits_R((A_6)_R)$ can be factored in the form $\footnotesize\begin{pmatrix}d^2 & 0\\c & d\end{pmatrix}=\footnotesize\begin{pmatrix}d^2 & 0\\0 & d\end{pmatrix}\footnotesize\begin{pmatrix}1 & 0\\c/d & 1\end{pmatrix}$. There are obvious group isomorphisms 
from $\left\{\footnotesize\begin{pmatrix}d^2 & 0\\0 & d\end{pmatrix}\colon d\in R^\times\right\}$ onto $ R^\times $ and from $\left\{\footnotesize\begin{pmatrix}1 & 0\\x & 1\end{pmatrix}\colon x\in R\right\}$ onto $ (R,+)$, which induce an isomorphism $\aut_R((A_6)_R)\to R^\times\times R$. Recall the affine line group, that is, the group of all maps $R\to R$ of the form $x\mapsto dx+c$, where $d\in R^\times$ represents a homotecy and $c\in R$ a translation. In this group, each element is a homotecy followed by a translation, so we can see that the affine line group is isomorphic to $R^\times\times R$ and, seen as an affine group scheme, is nothing but $\Gm\ltimes\Ga$. The notation $\ltimes$ refers to the fact that $\Ga$, seen as the subgroup of translations in the affine line group, is a normal subgroup (the $\ltimes$ symbol is \lq\lq open\rq\rq\ pointing towards the normal subgroup).

\subsection{The $\text{A}_\text{7}$ algebra}
For this algebra, the natural basis $\{e_1,e_2\}$ is such that $e_1^2 = e_1$ and $e_2^2 = 0$. Following the same reasoning as we did previously, finding an automorphism of $(A_7)_R$ translates into solving the following system of equations $ae_1+be_2 = a^2 e_1, 0 = c^2e_1, 0 = ac e_1, (ad-bc)u  = 1$. This implies that $a^2=a$, $b = 0$, $c^2 = 0$, $ac = 0$ and $adu = 1$. The last equation implies $a,d \in R^\times$ and hence $c = 0$ and $a = 1$. So, $\aut\nolimits_R((A_7)_R) \cong \left \{ {\footnotesize\begin{pmatrix}
        1 & 0 \\
        0 & d
    \end{pmatrix}} \colon d \in R^\times \right \}$. This description implies that $ \aut\nolimits_R((A_7)_R) \cong \Gm(R)$. Another way of seeing this is the following: since $(A_7)_R=Re_1 \oplus Re_2$, $f(e_1) $ is idempotent, and $Re_2={\rm ann}((A_7)_R)$ (which is fixed by automorphisms), we conclude that $f(e_2)\in Re_2$ and $f(e_1)=e_1$. The corresponding representing algebra of $\Gm$ is well known and can be found in \cite[p. 9]{Waterhouse}.

\subsection{The $\boldsymbol{\text{A}_{8,\alpha}}$  algebra. }
For this algebra, the natural basis $\{e_1,e_2\}$ is such that $e_1^2 = e_1$ and $e_2^2 = \a e_1$ with $\a \in K^\times$. In this case, finding the automorphisms is equivalent to solving the equations $ ae_1+be_2 = (a^2+\a b^2)e_1, \a ae_1+\a be_2  = (c^2 + \a d^2)e_1, 0 = (ac+\a bd) e_1, (ad-bc)u  = 1$. This system  of equations  is equivalent to $a = a^2$, $b = 0$, $\a a = c^2 + \a d^2$, $ac= 0$ and $adu = 1$. The last equation implies that $a,d \in R^\times$ and, together with $a^2 = a$, we have $a = 1$ and $c = 0$. In this way, we have $\a = \a d^2$ and hence $d^2 = 1$. So,  $\aut\nolimits_R((A_{8,\a})_R) \cong \left \{ {\footnotesize\begin{pmatrix}
        1 & 0 \\
        0 & d
    \end{pmatrix}} \colon d \in R, d^2 = 1 \right \}$. Moreover, we have that $\aut\nolimits_R((A_{8,\a})_R) \cong \Mu_2(R)$. In this case, the corresponding Hopf algebra is $\H_{8,\a} = K[x]/(x^2-1)$ with $\Delta\colon \H_{8,\a} \to \H_{8,\a}\otimes \H_{8,\a}$ given by $\Delta(\Bar{x}) = \bar x\otimes \Bar x$, $S \colon \H_{8,\a} \to \H_{8,\a}$ defined as $S(\bar x) = \bar x$ and $\varepsilon\colon \H_{8,\a} \to K$ with $\varepsilon(\bar x) = 1$. 
Notice that if $\text{char}(K)=2$, then $\text{aut}(A_{8,\alpha})=1$. Additionally, in the case of characteristic $2$, we have $\mathcal{H}_{8,\alpha}\cong K[x]/(x-1)^2\cong K[x]/(x^2)$, which is the algebra of dual numbers $K(\epsilon):=K\oplus K\epsilon$, a $2$-dimensional algebra with $\epsilon^2=0$.

The results of this section are summarized in the Table~\ref{table2}.

\begin{table}[ht]
\begin{center}
\setcellgapes{2pt}
\makegapedcells
\settowidth\rotheadsize{Perfect}
\settowidth\rotheadsize{Non-Perfect}
\begin{tabular}{|C{0.5cm}|c|c|c|}
\hline 
 & $\boldsymbol{A}$ & $\boldsymbol{\aut(A)} $ &   \bf{Hopf algebra}\\
   \hline
  &           $A_1$ & $\mu_2(K)$  & $K^2$\\
   \cdashline{2-4}
              
  & $A_{2,1}$ & $\mu_3(K)\rtimes C_2(K)$ &   $KC_3 \otimes K^2$\\
   \cdashline{2-4}
  
  & $A_{3,\a}$ &  $1$   &  $K$\\
   \cdashline{2-4}

  & $A_{4,\a}$ &  $1$  &$K$\\
   \cdashline{2-4}

  & $A_{5,\a,\b}$, \tiny ($\a\ne\b$) &  $1$   &$K$\\
   \cdashline{2-4}

  & $A_{5,\a,\a}$ & $\m_2(K)$ & $K^2$\\
   \hline 
\multirow[t]{-5.5}{*}{\rothead{Perfect}}

&   $A_5,\ \tiny\car=2$ & $(K,+)$  & $K[x]$\\ 
  \cdashline{2-4}
   
 &  $A_5,\ \tiny\car\ne 2$ & $(K,\cdot)$  & $K[x^{\pm}]$\\
    \cdashline{2-4}
   
  & $A_6$ & $K^\times \times K$ & $K[x,y^{\pm}]$\\
    \cdashline{2-4}
   
  & $A_7$ & $K^\times$ & $K[x^{\pm}]$\\
   \cdashline{2-4}
   
  & $A_{8,\a},\ \tiny\car\ne 2$ & $\mu_2(K)$ & $K^2$ \\
   \cdashline{2-4}
   \multirow[t]{-4}{*}{\rothead{Non-Perfect}}
  & $A_{8,\a},\ \tiny\car= 2$ & $1$ &  $K(\epsilon)=\frac{K[x]}{(x^2)}$\\
  \hline

\end{tabular}
\caption{\small Hopf algebras of $\aut(A) $.}
\label{table2}
\end{center}
\end{table}
\medskip

\section{Associative representations of two-dimensional evolution algebras}\label{morrocotudo}

In this section, we describe associative and commutative representations of the algebras studied in the previous section. The main idea that we pursue is that the algebra $\U_p$ of the universal associative and commutative representation of an evolution algebra $A$ (relative to $p$) and the Hopf algebra representing the affine group scheme $\affaut(A)$ are related. To reveal this relation, we study some representations before we set up the notation.

\begin{definition}\label{fed} \rm
 Let $A$ be a two dimensional evolution algebra with natural basis $B=\{e_1,e_2\}$ such that
$e_i^2=\w_{1i} e_1+\w_{2i}e_2$ ($i=1,2$). Fix an element $p=\l_0 ab+\l_1 ab^*+\l_2 a^*b+\l_3 a^*b^*$ in the polynomial algebra with involution in two indeterminates $a$ and $b$. We define $\U_p:=K[x,y,x^*,y^*]/I$, where $I$ is $*$-generated by the polynomials $p(x,x)-\w_{11}x-\w_{21}y$, $p(y,y)-\w_{12}x-\w_{22}y$, $p(x,y)$ and $p(y,x)$.  The elements of $\mathcal{U}_p$ are denoted by $\overline{z}$, where $z\in K[x,y,x^*,y^*]$, and the product in $\mathcal{U}_p$ is given by  $\bar z \bar w = \overline{p(z,w)}$. The homomorphism $\rho\colon A\to\U_p$ such that $e_1\mapsto \bar x$ and $e_2\mapsto \bar y$ will be called the  {\textit {universal associative and commutative representation}} of $A$ (notice that it depends on the chosen polynomial $p$ and the natural basis $B$).
\end{definition}

We will use the symbol $\equiv$ to denote the equivalence relation in $\U_p$ of congruence modulo $I$. 
Notice that the polynomial $p$ in Definition \ref{fed} can be evaluated in any associative unital algebra with involution. If $(B,*)$ is such an algebra, then for any
$\sigma\colon A\to B$ such that $\sigma(uv)=p[\sigma(u),\sigma(v)]$ with $u,v\in A$, there is a unique $*$-homomorphism of unital associative algebras $F\colon\U_p\to B$ such that $F\rho=\sigma$. This universal associative and commutative representation $\rho$ is usually identified with the $*$-algebra $\U_p$, although it should be noted that this identification is an abuse of notation. Also, notice that this kind of universal representation is unital by birth, in contrast with the notion used in \cite{squares}. The reason for taking an unital algebra version is that we are interested in using some scheme-theoretic notions, for which it is essential to work on the category of associative, commutative, and unital algebras. However, it will also be useful to consider a non-unital version of $\U_p$. Define $\T_p$ as the quotient of the free associative and commutative $K$-algebra generated by $\{x,y,x^*,y^*\}$ by the ideal $I$ above. Thus, while $\U_p$ is unital, we do not have, a priori, a unit in $\T_p$. Furthermore, $\T_p$ can be seen as a subalgebra of $\U_p$ and $I\subset \T_p$ so that, we can consider the restriction
$\rho\colon A\to\T_p$ (recall $e_1\mapsto\bar x$ and $e_2\mapsto\bar y$). This associative and commutative representation has the property that for any other associative algebra with involution $(B,*)$ (not necessarily unital), and $\s\colon A\to B$ satisfying
$\s(uv)=p[\s(u),\s(v)]$ for $u,v\in A$,
there is a unique $*$-homomorphism of associative algebra $F\colon\T_p\to B$ such that $F\rho=\s$.
\begin{definition}\rm 
In the sequel, to shorten terminology, the unital associative and commutative algebra $\U_p$ associated to the evolution algebra $A$ will be called {\it universal unital $p$-algebra} of $A$. The (not necessarily unital) version $\T_p$  will be termed the {\it tight universal $p$-algebra} of $A$. If the representation $A\to\U_p$  (respectively $A\to\T_p$) is faithful, we will say that $\U_p$ is a {\it faithful universal unital $p$-algebra} (resp.  {\it faithful tight universal $p$-algebra)}.
\end{definition}

\begin{lemma}\label{leonsinchuleton}
Let $\E$ be a two-dimensional evolution algebra.

\begin{enumerate}[\rm (i)]
    \item If $\w_{i2}=0$ for $i=1,2$, then $\E$ is the zero-product algebra or isomorphic to $A_6$ or $A_7$.
    \item If $\w_{11}\ne 0$ and $\w_{21}=\w_{12}=0$, then $\E$ is isomorphic to $A_1$ or to $A_7$.
\end{enumerate}
\end{lemma}

\begin{proof}
    
For item (i), if $\w_{11}=0$ then $e_1^2=\w_{21} e_2$ and $e_2^2=0$.  Thus, if the product in $\E$ is non-trivial, we have that $\E\cong A_6$. Assume next that $\w_{11}\ne 0$. Then, $\{e_1^2,e_2\}$ is linearly independent. Indeed, 
from $\l e_1^2+\m e_2=0$ we get that
$\l\w_{11}e_1+(\l\w_{21}+\mu)e_2=0$ and hence $\l=\m=0$. This new basis is natural since $e_1^2 e_2=0$, and the multiplication relative to it is  
$(e_1^2)^2=(\w_{11}e_1+\w_{12}e_2)^2=(\w_{11})^2e_1^2$ and $e_2^2=0$. So, in this case, $A$ is isomorphic to $A_7$. 

For item (ii), let $k:=\frac{1}{\w_{11}}$. Since $e_1^2=\w_{11} e_1$, we obtain that $(ke_1)^2=ke_1$. So, we may assume without loss of generality that $e_1$ is idempotent. Now, if $\w_{22}\ne 0$, we may also assume that $e_2^2=e_2$ and we have that $\E$ is isomorphic to $A_1$. If $\w_{22}=0$ then $\E$ is isomorphic to $A_7$. 
\end{proof}

Next, we give some purely algebraic results about ideals in polynomial algebras that we will use in the sequel.
\begin{lemma}\label{hell} Let $I\triangleleft K[u,v]$ be generated by $uv$, $u^2-au-bv$, $v^2-cu-dv$ where $a,b,c,d\in K$.
Consider the $K$-algebra $R:=K[u,v]/I$.
Then, denoting by $\bar z$ the equivalence class of $z$ in $R$, we have that $\bar u$ and $\bar v$ are linearly dependent elements in $R$ except in the cases: 
\smallskip (i) $b\ne 0$, $c=d=0$,
\smallskip (ii) $b=a=0.$
\smallskip (iii) $b=0$, $a\ne 0$, $c=0$.
\end{lemma}
\begin{proof} Since $uv\in I$ and $u^2-au-bv\in I$ we have that $b v^2\in I$. If $b\ne 0$ then $v^2\in I$, implying $cu+dv\in I$. So $\bar u$ and $\bar v$ are linearly dependent unless $c=d=0$. We analyze this case: we have $I=(uv, u^2-au-bv, v^2)$ and so, if $\l u+\m v\in I$ for some scalars $\l,\m\in K$, we have that
$$\l u+\m v=p(u,v)uv+q(u,v)(u^2-au-bv)+r(u,v)v^2$$
in the polynomial algebra $K[u,v]$. Putting $v=0$, we get that
$\l u=q(u,0)(u^2-au)$.

 Since $(u^2-au)$ has degree 2, $q(u,0)$ is a polynomial in $u$, and $\lambda u$ has a degree less or equal to 1, we have that $q(u,0)$ must be zero. Hence, $\lambda u =0 $ and $\lambda =0$.

Since $q(u,0)=0$, we have that $q(u,v)=v\xi(u,v)$ for a certain polynomial $\xi(u,v)$. Consequently, 
$$\m v=p(u,v)uv+v \xi(u,v)(u^2-au-bv)+r(u,v)v^2 \text{ and so }$$
$$\m =p(u,v)u+ \xi(u,v)(u^2-au-bv)+r(u,v)v.$$
Making $u=v=0$ we get that $\mu=0$. Thus, $\bar u$ and $\bar v$ are linearly independent in this case.

Next, we study the case $b=0$. In such case, $I=(uv, u^2-a u, v^2-cu-dv)$. If $a=0$ then $I=(uv,u^2,v^2-cu-dv)$. Assume that 
$$\l u+\m v=p(u,v)uv+q(u,v) u^2+r(u,v)(v^2-cu-dv).$$
For $u=0$ we get that $\m v=r(0,v)(v^2-dv)$, which gives $\m=0$ and $r(0,v)=0$. Also $r(u,v)=u\xi(u,v)$ for some polynomial $\xi(u,v)$. Whence $\l=p(u,v)v+q(u,v)u+\xi(u,v)(v^2-cu-dv)$,
which gives $\l=0$ (putting $u=v=0$). In this case, $\bar u$ and $\bar v$ are linearly independent. If $a\ne 0$, since
$I=(uv, u^2-a u, v^2-cu-dv)$, we have that $c u^2\in I$ and hence $c a u\in I$ and
$c u\in I$. In case $c\ne 0$, we have that $u\in I$. So, $\bar u=0$ and $\bar v$ are linearly dependent. If $c=0$, we have that
$I=(uv, u^2-a u, v^2-dv)$ and it is easily seen, following the previous idea, that $\bar u$, $\bar v$ are linearly independent. 
\end{proof}

\begin{proposition}\label{rose}
Suppose that $x\equiv x^*$ and $y\equiv y^*$. If the universal representation $\rho\colon A\to\U_p$ is faithful, then $A\cong A_1, A_6$ or $A_7$ (i.e., $A$ is associative).
\end{proposition} 
\begin{proof}
The ideal $I$ is generated by 
$x-x^*$, $y-y^*$, $\l xy$, $\l x^2-\w_{11} x-\w_{21} y$ and 
$\l y^2-\w_{12} x-\w_{22} y$,
where $p(a,b)=\l_0 ab+\l_1 ab^*+\l_2a^*b+\l_3a^*b^*$ and $\l:=\sum\l_i$. We analyze two cases:\par
\item{(a)} $\l=0$. Then $\w_{11} x+\w_{21}y\in I$ and $\w_{12} x+\w_{22}y\in I$. Ruling out the algebra of zero product (in which all the $\w_{ji}=0$), we have that some $\w_{ji}\ne 0$. Thus, either 
    $\w_{11}x+\w_{21}y\in I$ with some scalar $\w_{11},\w_{21}$ nonzero, or 
    $\w_{12}x+\w_{22}y\in I$ with some scalar $\w_{12},\w_{22}$ nonzero. In any case, $\bar x$ and $\bar y$ are linearly dependent elements of $\U_p$, and the universal associative $*$-representation of $A$ is not faithful. 
 \item{(b)} $\l\ne 0$. Then 
 $xy\in I$,
 $x^2-\l^{-1}\w_{11} x-\l^{-1}\w_{21} y\in I$ and 
$y^2-\l^{-1}\w_{12} x-\l^{-1}\w_{22} y \in I$.
So, $\U_p=K\bar 1+K\bar x+K\bar y$ with 
$\bar x\bar y=0$ and $\bar x^2=\l^{-1}\w_{11} x+\l^{-1}\w_{21} y$,
$\bar y^2=\l^{-1}\w_{12} x+\l^{-1}\w_{22} y$. Next, we apply Lemma~\ref{hell} with $a=\l^{-1}\w_{11}$, $b=\l^{-1}\w_{21}$, $c=\l^{-1}\w_{12}$ and $d=\l^{-1}\w_{22}$. We conclude that $\bar x$ and $\bar y$ are linearly dependent except in the cases:
\begin{enumerate}[\rm (i)]
\item  $\w_{21}\ne 0$, $\w_{12}=\w_{22}=0$. Then, $A\cong A_6$.
\item $\w_{11}=\w_{21}=0$. Then, if $A$ is not a zero-product algebra, $A\cong A_6$ or $A\cong A_7$.
\item $\w_{11}\ne 0$, $\w_{21}=\w_{12}=0$. Then $A\cong A_1$ or $A\cong A_7$,
\end{enumerate}
where to identify $A$ we have applied Lemma~\ref{leonsinchuleton}. 
 \end{proof}

 \begin{remark}\rm 
If $x\equiv x^*$ and $\w_{21}\ne 0$, then $y\equiv y^*$. This is because $\l x^2-w_{11} x-w_{21} y\in I$, where $\l=\sum_{i=0}^3 \l_i$. Applying $*$ we get that $\l x^2-w_{11}x-w_{21} y^*\in I$ and so, if $\w_{21}\ne 0$, we conclude that
$y\equiv y^*$. Similarly, we have that if $y\equiv y^*$ and $\w_{12}\ne 0$ then $x\equiv x^*$.
\end{remark}

\begin{lemma}\label{eqp}Assume that $\E$ is a perfect two-dimensional evolution algebra with a faithful universal unital $p$-algebra $\U_p$ such that $x^2 \equiv (x^*)^2$ and $y^2 \equiv ( y^*)^2$ in $\U_p$. 
Then $A\cong A_1, A_6$ or $A_7$ (i.e., $A$ is associative). Consequently, the Hopf algebra $\H$ representing the group scheme $\aut(\E)$ satisfies $\H\ne K$. 
\end{lemma}
\begin{proof} Consider $p(a,b)=\l_0 ab+\l_1 ab^*+\l_2a^*b+\l_3a^*b^*$. Since $p(x,x)-\w_{11}x-\w_{21} y \in I$, we have that
\begin{align*}
\l_0x^2+(\l_1+\l_2) xx^*+\l_3(x^*)^2-\w_{11}x-\w_{21}y&\in I \text{ and }\cr
\l_3x^2+(\l_1+\l_2) xx^*+\l_0(x^*)^2-\w_{11}x^*-\w_{21}y^*&\in I.
\end{align*}
So, we have that $(\l_0-\l_3)(x^2-(x^*)^2)+\w_{11}x^*+\w_{21}y^*-\w_{11}x-\w_{21} y\in I$ and, from the fact that $p(y,y)-\w_{12}x-\w_{22} y \in I$, we get that $(\l_0-\l_3)(y^2-(y^*)^2)+\w_{12}x^*+\w_{22}y^*-\w_{12}x-\w_{22} y\in I$. Since 
$x^2-(x^*)^2$ and $y^2-(y^*)^2\in I$ we obtain that 
\begin{align}
\w_{11}x^*+\w_{21} y^* \equiv &\ \w_{11} x+\w_{21} y \text{ and }\cr
\w_{12}x^*+\w_{22} y^* \equiv &\  \w_{12} x+\w_{22} y.
\end{align}
Since $\E$ is perfect, the matrix $(\w_{ji})$ is invertible and hence $x^*\equiv x$ and $y^*\equiv y$. Now, applying Proposition~\ref{rose}, we get the claimed conclusion.

\end{proof}

Next, we compute the $\U_p$ and $\T_p$ algebras associated to perfect evolution algebras $A$ of dimension $2$ and relate them with $\H$, the representing Hopf algebra of the automorphism group scheme of $A$. We also compare
the groups $\aut(A)$ and $\aut^*(\T_p)$ in all cases.

\subsection{The $\text{A}_\text{1}$ algebra}
We proved before that the Hopf algebra representing $\affaut(A_1)$ is isomorphic to $\H_1=K[x]/(x^2-x)$. This algebra is $2$-dimensional and isomorphic to $K^2$ with componentwise operations. Consider the identity involution in this algebra and the product $p(a,b)=ab^*$. It is easy to verify that the map
$\rho\colon A_1\to\H_1$ given by $e_1\mapsto \bar x$ and $e_2\mapsto \bar 1-\bar x$ is a faithful associative and commutative representation of $A_1$ in $\H_1$, relative to $p$ and the identity involution on $\H_1$. Consequently, the universal associative and commutative representation of $A_1$, with the product $(a,b){\buildrel{p}\over\mapsto} ab^*$, is faithful. Let us prove that $\U_p$ can be identified with
$K[x,y]/I$, where $I=(x^2-x,y^2-y,xy)$. 
First, let us check that $\{\bar 1,\bar x,\bar y\}$ is a $K$-basis of $K[x,y]/I$. We begin proving that $x\notin I$. If $x\in I$ then we have an identity in $K[x,y]$ of the form $x=p(x^2-x)+q(y^2-y)+r xy$, for some polynomials $p,q,r$. Then, making $x=0$, we get that $0=q(0,y)(y^2-y)$ and hence
$q(0,y)=0$, so that $q(x,y)=x q_1(x,y)$ for some other polynomial $q_1$. But then 
$1=p(x-1)+q_1(y^2-y)+ry$ and, for $y=0$, we get $1=p(x,0)(x-1)$, which is impossible in the polynomial ring. In a similar way, $y\notin I$.
Now we prove that $1-x-y\notin I$: on the contrary $1-x-y=p(x^2-x)+q(y^2-y)+r xy$ again for some polynomials $p,q$ and $r$. But, making $y=0$ we get that $1-x=p(x,0)(x^2-x)$ and therefore $-1=p(x,0)x$, which is also impossible. Now, we prove that $\{\bar 1,\bar x,\bar y\}$ is $K$-linearly independent: if $\a\bar 1+\b\bar x+\gamma\bar y=0$
for scalars $\a,\b,\gamma\in K$,
then $(\a+\b)\bar x=0$ and hence $\a+\b=0$ because $\bar x\ne 0$. Similarly, $(\a+\gamma)\bar y=0$, which implies that $\a+\gamma=0$. So, $\b=\gamma=-\a$. If
$\a=0$ then $\b=\gamma=0$. On the contrary,
$\bar 1-\bar x-\bar y=0$ so that $1-x-y\in I$, which is not possible. 
Now that we have that $K[x,y]/I$ is a $3$-dimensional algebra with basis the classes of $1, x$ and $y$, let us prove that we can identity $\U_p$ with $K[x,y]/I$. 
In fact, we define $\rho\colon A_1\to K[x,y]/I$ by $(1,0)\mapsto\bar x$ and $(0,1)\mapsto\bar y$. This satisfies $\rho(st)=p[\rho(s),\rho(t)]$ ($s,t\in A_1$) and it is a monomorphism. Given any other associative and commutative unital algebra with involution $(B,*)$, and $\sigma\colon A_1\to (B,*)$ satisfying $\sigma(st)=p[\s(s),\s(t)]$ for any $s,t\in A_1$, we can define $F\colon K[x,y]/I\to B$ by $F(\bar x)=\sigma(1,0)$ and $F(\bar y)=\sigma(0,1)$.
Notice that for any $z\in A_1$ we have $\s(z)=\s(z\cdot 1)=\s(z)\s(1)^*$ but also $\s(z)=\s(1\cdot z)=\s(1)\s(z)^*$. Therefore,
$\s(z)^*=(\s(1)\s(z)^*)^*=\s(z)\s(1)^*=\s(z)$ for any $z\in A_1$. Then, it is not difficult to prove that $F\colon K[x,y]/I\to B$ is a homomorphism of associative algebras. Moreover, if we endow $K[x,y]/I$ with the identity involution, we have
$F(\bar x^*)=F(\bar x)=\s(1,0)=\s(1,0)^*=F(\bar x)^*$ and in a similar fashion $F(\bar y^*)=F(\bar y)^*$. Thus, $F$ is a $*$-homomorphism
$K[x,y]/I\to (B,*)$,  and it is the unique one that satisfies $F\rho=\sigma$. Therefore, from this point on, we identify $\U_p$ with $K[x,y]/I$. Modulo this identification,  $\T_p=K\bar x\oplus K\bar y$, where we know
that $\bar x$, $\bar y$ are orthogonal idempotents. Moreover, $\bar x+\bar y$ is the unit of $\T_p$. Hence, $\T_p\cong A_1\cong\H_1$ (as algebras), and we have proved the following (natural) result.
\begin{proposition}
    The tight universal $p$-algebra of $A_1$ is $A_1$ itself. Furthermore,  $\T_p\cong A_1\cong\H_1$ and 
    $\aut^*(\T_p)=\aut(\T_p)\cong\aut(A_1)$ is the cyclic group of order $2$.
\end{proposition}
As an obvious consequence, the representing algebra of $\affaut(\T_p)$ agrees with that of $\affaut(A_1)$ and $A_1$ admits a tight universal $p$-algebra which is $\H_1$.

 For the next subsection we should take into account that any representation of an evolution $K$-algebra $\mathcal E$, namely, $\rho\colon \mathcal{E}\to A$ (where $A$ is an associative commutative $K$-algebra with involution), gives rise to a representation $\rho_F\colon \mathcal{E}_F\to A_F$ of the extended $F$-algebra $E_F$ relative to any extension field $K\subset F$. Moreover if $\rho$ is a monomorphism then $\rho_F$ is it also. But also, if we have an evolution $K$-algebra $\mathcal{E}$ such that a certain scalar extension $\mathcal{E}_F$ has a representation $\sigma\colon\mathcal{E}_F\to B$ (where $B$ is an $F$-algebra with involution), then the composition
$\mathcal{E}\to B$ such that $x\mapsto\sigma(x\otimes 1)$ is a representation on $B$ (and it is faithfull if $\sigma$ is). This justifies jumping to the algebraic closure eventually.

\subsection{The $\boldsymbol{\text{A}_{2,\alpha}}$ algebra}
For this algebra, the natural basis $\{e_1,e_2\}$ satisfies $e_1^2=e_2$ and $e_2^2=\a e_1$, with $\a \in K^\times$. Extending scalars, if necessary,  we may take $\a=1$ because the elements $e_1':=\root 3\of {\a^{-1}} e_1$ and
$e_2':=\root 3\of {\a^ {-2}}e_2$, satisfy the multiplication relations of $A_{2,1}$. Thus, we consider the algebra $A_{2,1}$ from the beginning and study its associative representations. In previous sections, we observed an involution in the Hopf algebra $\H_{2,1}$ such that $a^*=b^2$ and $b^*=a^2$. If we consider the product $\H_{2,1}\times \H_{2,1}\to\H_{2,1}$ such that $(x,y)\mapsto x^*y^*$, then we have a monomorphism $\rho\colon A_{2,1}\to (\H_{2,1},p)$ induced by $\rho(e_1)=a$ and $\rho(e_2)=b$. Thus, the universal unital $p$-algebra $\U_p$ (relative to $p\colon (x,y)\mapsto x^*y^*$) is faithful and $\H_{2,1}$ is a quotient of $\U_p$. For this algebra, we have that $\U_p=K[x,y,x^*,y^*]/I$, where $I=(x^2-y^*,y^2-x^*,xy)_*$. Moreover, $\U_p:=K[x,y,x^*,y^*]/I\cong K[x,y]/(x^4-x,y^4-y,xy)$.
Notice that a $K$-basis of $\U_p$ is given by
$\{\bar 1,\bar x,\bar x^2,\bar x^3,\bar y,\bar y^2,\bar y^3\}$, which implies that $\dim(\U_p)=7$. The tight universal $p$-algebra $\T_p$ is the $K$-linear span
of $\mathcal{B}=\{\bar x,\bar x^2,\bar x^3,\bar y,\bar y^2,\bar y^3\}$, which is $6$-dimensional and unital (with unit $\bar x^3+\bar y^3$). 
The involution $*$ of the tight $p$-algebra $\T_p$ is the induced by $\bar x^*=\bar y^2$. It can be checked that the matrix of $*$ relative to $\mathcal{B}$ is $\small\begin{pmatrix} 0 & m\\ m & 0\end{pmatrix}$, where $m=\small \begin{pmatrix}0 & 1 & 0\\1 & 0 & 0\\0 & 0 &1\end{pmatrix}.$

Recalling the algebra $\H_{2,1}$, we see that $\T_p\cong\H_{2,1}$ and $\U_p$ is just the unitization of $\T_p$. If $K$ is algebraically closed and $\car(K)\ne 3$ then, for the group $\aut(A_{2,1})$, we have 
$\aut(A_{2,1})=\{1,\bar\omega,\bar\omega^2,\s,\s\bar\omega,\s\bar\omega^2\}$, where $\bar\omega=\hbox{Diag}(\omega,\omega^2)$ and 
$\s=\tiny\begin{pmatrix}0 & 1\\ 1 & 0\end{pmatrix}$, where $\omega$ is a primitive cubic root of $1$. We can identify $\aut(A_{2,1})$ (as an abstract group) with the symmetric group $S_3$. Since $\rho$ is faithful, we can embed $\aut(A_{2,1})$ into $\aut(\T_p)$. 

The subgroup $\aut^*(\T_p)$ of $*$-automorphisms admits the matrix representation 
$$\left \{\begin{pmatrix}a & 0\\0 & mam\end{pmatrix}\colon a\in S_3 \right \}\sqcup \left \{\begin{pmatrix}0 & b\\mbm & 0\end{pmatrix}\colon b\in S_3 \right \}$$
relative to the basis $\mathcal{B}$ (using block matrices). Thus, we have an isomorphism $\aut^*(\T_p)\cong \aut(A_{2,1})\rtimes \mathbb{Z}_2$.

\begin{theorem}
For the $A_{2,1}$ algebra, the tight universal $p$-algebra $\T_p$ is isomorphic to the representing Hopf algebra of $\affaut(A_{2,1})$, which is $\H_{2,1}$. The group $\aut^*(\T_p)$ can be writen as
$\aut^*(\T_p)\cong \aut(A_{2,1})\rtimes \mathbb{Z}_2$.
\end{theorem}
Thus, in this case, $A_{2,1}$ admits a tight universal $p$-algebra, which is nothing but $\H_{2,1}$, the Hopf algebra representing its group of automorphisms.

\subsection{The $\boldsymbol{\text{A}_{3,\alpha}}$  algebra}
In the work \cite{squares}, we found some computational evidence (but not completely conclusive) that this algebra has no faithful associative and commutative representation. In this subsection, we prove that this is indeed the case. 
For this algebra, the natural basis $\{e_1,e_2\}$ satisfies $e_1^2=e_1$ and $e_2^2=\a e_1+e_2$, with $\a \in K^\times$. Consider a product $p(a,b)=\l_0 ab+\l_1 ab^*+\l_2a^*b+\l_3 a^*b^*$ with $\l_i\in K$ and the universal associative and commutative representation given by $\U_p=K[x,y,x^*,y^*]/I$. In this case
$I$ is the ideal $I=(p(x,x)-x,p(y,y)-\a x-y,p(x,y),p(y,x))_*$. Thus, $I$ is generated by the following polynomials.
$$\begin{matrix*}[l]
 p(x^*,x^*)-x^* &  =\lambda _3 x^2+\lambda _1 x^* x+\lambda _2 x^* x+\lambda _0 \left(x^*\right)^2-x^*,\\
 p(x,x)-x& = \lambda _0 x^2+\lambda _1 x^* x+\lambda _2 x^* x+\lambda _3 \left(x^*\right)^2-x,\\
 p(y^*,x^*)  & = \lambda _0 x^* y^*+\lambda _2 x^* y+\lambda _1 x y^*+\lambda _3 x y,\\
 p(x^*,y^*) & =\lambda _0 x^* y^*+\lambda _1 x^* y+\lambda _2 x y^*+\lambda _3 x y,\\
 p(x,y)& = \lambda _3 x^* y^*+\lambda _2 x^* y+\lambda _1 x y^*+\lambda _0 x y,\\
 p(y,x) & = \lambda _3 x^* y^*+\lambda _1 x^* y+\lambda _2 x y^*+\lambda _0 x y,\\
 p(y^*,y^*)-\alpha x^*-y^*& = -\alpha  x^*+\lambda _3 y^2+\lambda _1 y^* y+\lambda _2 y^* y+\lambda _0 \left(y^*\right)^2-y^*,\\
 p(y,y)-\alpha x-y &= -\alpha  x+\lambda _0 y^2+\lambda _1 y^* y+\lambda _2 y^* y+\lambda _3 \left(y^*\right)^2-y.
\end{matrix*}$$

Let us analyze first the case $\l_0=\l_3$. Subtracting the first two equations, we get $x\equiv x^*$; subtracting the last two equations, we get $y\equiv y^*$. Then, Proposition~\ref{rose} implies that $\rho$ is not faithful (given that $A_{3,\a}$ is not associative).
So, we proceed assuming that $\l_0\ne\l_3$. The subtraction 3rd - 5th polynomial gives $x^*y^*\equiv xy$ and the subtraction 3rd-4th polynomial gives $(\l_2-\l_1)(x^*y-xy^*)$.
Consequently, we discuss according to the dichotomy:

\begin{center}\fbox{Case $\l_2\ne\l_1$}\end{center}
In this case, we have that $x^*y\equiv xy^*$ and $x^*y^*\equiv xy$. A  set of elements of the ideal $I$ is:
$$\begin{matrix*}[l] 
l_1= x y-x^* y^*\\
l_2=x y^*-x^* y,\\
l_3=\lambda _3 x^* y^*+\lambda _2 x^* y+\lambda _1 x y^*+\lambda _0 x y,\\
l_4=\lambda _3 x^* y^*+\lambda _1 x^* y+\lambda _2 x y^*+\lambda _0 x y,\\
l_5=\lambda _0 x^2+\lambda _1 x^* x+\lambda _2 x^* x+\lambda _3 \left(x^*\right)^2-x,\\
l_6=\lambda _3 x^2+\lambda _1 x^* x+\lambda _2 x^* x+\lambda _0 \left(x^*\right)^2-x^*,\\
l_7=-\alpha  x+\lambda _0 y^2+\lambda _1 y^* y+\lambda _2 y^* y+\lambda _3 \left(y^*\right)^2-y,\\
l_8=-\alpha  x^*+\lambda _3 y^2+\lambda _1 y^* y+\lambda _2 y^* y+\lambda _0 \left(y^*\right)^2-y^*.\\

\end{matrix*}$$

Now, define the polynomials $I\ni a_1:=y^*l_1-y l_2=
y^*(x y-x^*y^*)-y(xy^*-x^*y) = x^*y^2-x^*(y^*)^2$. Also,
$I\ni a_3:=\l_0 l_1-l_3=\l_0(xy-x^*y^*)-(\lambda _3 x^* y^*+\lambda _2 x^* y+\lambda _1 x y^*+\lambda _0 x y)\equiv -(\l_0+\l_3)x^*y^*-(\l_2+\l_1)x^*y$. Then
$I\ni a_4:=y^* a_3+x^* l_8\equiv 
-y^*((\l_0+\l_3)x^*y^*+(\l_2+\l_1)x^*y)+x^*l_8\equiv -\a(x^*)^2-x^*y^*$. Consequently, $\a (x^*)^2+x^*y^* \in I$ and $\a x^2+xy\in I$ and hence, since $xy\equiv x^*y^*$, we obtain that
$x^2\equiv (x^*)^2$. Considering $l_5$ and $l_6$ we also get $x\equiv x^*$.

Therefore, the following elements are in $I$:
\begin{align*}
 & (1) \  xy-xy^*,  & (4)  & \  \l x^2-x, \text{where } \l=\sum_{i=0}^3\l_i, \\
& (2) \   (\l_3+\l_1) xy^*+(\l_0+\l_2)xy,   & (5) &  \ -\alpha  x+\lambda _0 y^2+\lambda _1 y^* y+\lambda _2 y^* y+\lambda _3 \left(y^*\right)^2-y,\\
  & (3) \   (\l_3+\l_2) xy^*+(\l_0+\l_1)xy,  & (6) & \  -\alpha  x^*+\lambda _3 y^2+\lambda _1 y^* y+\lambda _2 y^* y+\lambda _0 \left(y^*\right)^2-y^*.
\end{align*}

Observe that in the case $\l=0$, we have that $x\in I$, and we conclude that the universal associative representation is not faithful. So, we proceed assuming that $\l\ne 0$. Notice that
$$(xy^*,xy)\begin{pmatrix}
    \l_3+\l_1 & \l_3+\l_2\\
    \l_0+\l_2 & \l_0+\l_1
    \end{pmatrix}\in I\times I
,$$
and the determinant of the above matrix is $\left(\lambda _1-\lambda _2\right) \left(\lambda _0+\lambda _1+\lambda _2+\lambda _3\right)=\l(\l_1-\l_2)$. Recall that we are under the hypothesis that $\l_1\ne\l_2$. Therefore,  the above matrix is invertible, which implies that $xy^*, xy\in I$. Then, multiplying the polynomial in (5) by $x$ and taking into account (4), we get that $x\in I$ because $\a\ne 0$. Thus, we have again that the universal associative representation is not faithful.
\begin{center}\fbox{Case $\l_2 = \l_1 = 0$}\end{center} In case $\l_1=\l_2=0$, the ideal $I$ contains the following set of polynomials: 
\begin{align*}
    (a) & \ \lambda _3 x^2+\lambda _0 \left(x^*\right)^2-x^*, & (d) & \ \lambda _3 x^* y^*+\lambda _0 x y,\\ 
    (b) &\ \lambda _0 x^2+\lambda _3 \left(x^*\right)^2-x, &  (e) & \  -\alpha  x^*+\lambda _3 y^2+\lambda _0 \left(y^*\right)^2-y^*,\\
    (c) & \ \lambda _0 x^* y^*+\lambda _3 x y, & (f) & \  -\alpha  x+\lambda _0 y^2+\lambda _3 \left(y^*\right)^2-y.
\end{align*}

Since we have proved that $xy\equiv x^*y^*$, from (c) we obtain that $(\l_0+\l_3)xy\in I$. We distinguish two cases:

\begin{center}\fbox{Case $\l_2 = \l_1 = 0$, $\l_0+\l_3 \ne 0$}\end{center}
If $\l_0+\l_3\ne 0$ then we have that $xy\in I$ (also $x^*y^*\in I$) and hence $y(\lambda _0 x^2+\lambda _3 \left(x^*\right)^2-x) \in I$. Therefore, $\l_3 (x^*)^2y\in I$.  
In case $\l_3=0$ we have that the four polynomials below belong to $I$:
\begin{align*}
   & \lambda _0 \left(x^*\right)^2-x^*, & -\alpha  x^*+\lambda _0 \left(y^*\right)^2-y^*,\\
   & \lambda _0 x^2-x, & -\alpha  x+\lambda _0 y^2-y.
\end{align*}
Since $xy\in I$, we get $x^2\in I$ from the fourth polynomial, and so the second polynomial implies $x\in I$. Consequently, the universal associative representation is not faithful. If $\l_3\ne 0$ we have $(x^*)^2y\in I$ and $x^2y^*\in I$. Since we know that $\lambda _3 x^2+\lambda _0 \left(x^*\right)^2-x^*\in I$, multiplying this polynomial by $y$, we conclude that $x^*y\in I$ (so $xy^*\in I$). Finally, from the fact that $-\alpha  x+\lambda _0 y^2+\lambda _3 \left(y^*\right)^2-y\in I$, multiplying this polynomial by $x$, we get that $x^2\in I$ and whence $(x^*)^2\in I$. Thus, $x\in I$ and we reach the same conclusion: the universal associative representations is not faithful.
\medskip

\begin{center}\fbox{Case $\l_2 = \l_1 = 0$, $\l_0+\l_3=0$}\end{center}
Assume now that $\l_0+\l_3=0$. Then the following elements are in $I$:
\begin{align*}
    & \l_0[(x^*)^2-x^2]-x^*, & -\a x-y+\l_0[y^2-(y^*)^2], \\
   & \l_0[x^2-(x^*)^2]-x, & -\a x^*-y^*+\l_0[(y^*)^2-y^2].
\end{align*}
From the first two polynomials, we obtain that $x+x^*\in I$. But then, the second polynomial gives $x\in I$, and hence, the universal associative representation is not faithful.

\begin{center}\fbox{Case $\l_2=\l_1 \neq 0$}\end{center} Now, we must analyze the case $\l_2=\l_1\ne 0$. Without loss of generality we may take $\l_2=\l_1=1$. Then, the ideal $I$ contains the polynomials
\begin{align*}
   & p_1 = \lambda _3 x^2+2x^* x+\lambda _0 \left(x^*\right)^2-x^*, & &p_4 =-\alpha  x^*+\lambda _3 y^2+2y^* y+\lambda _0 \left(y^*\right)^2-y^*,& \\
   & p_2 =\lambda _0 x^2+2x^* x+\lambda _3 \left(x^*\right)^2-x, & & p_5 =-\alpha  x+\lambda _0 y^2+2y^* y+\lambda _3 \left(y^*\right)^2-y.& \\
   &  p_3 =(\lambda _0+\l_3) x y+x^* y+x y^*, &  & \\
\end{align*}
Recall that we have $xy\equiv x^*y^*$.
To improve readability, we put $z:=x^*$ and $t:=y^*$. The $p_i$ polynomials above 
are then written as:
\begin{align}\label{lor}
\begin{split}
p_1= &\lambda _3 x^2+2 x z+\lambda _0 z^2-z,\\ 
p_2= &\lambda _0 x^2+2 x z-x+\lambda _3 z^2,\\
p_3= &t x+\left(\lambda _0+\lambda _3\right) x y+y z,\\
p_4= &\lambda _0 t^2+2 t y-t+\lambda _3 y^2-\alpha  z,\\
p_5= &\lambda _3 t^2+2 t y-\alpha  x+\lambda _0 y^2-y.
\end{split}
\end{align}
Now, recall that $xy-tz\in I$ and notice that 
\begin{align}
\begin{split} 0 = t^2 p_1 +(x y-t z)p_3 - t y p_2- x^2p_5+x zp_4 + \\
(x y-t z) 
\left[\lambda _0 (t x+t z)+t x-t-\lambda _3 (x y+y z)-x-y z\right]+\alpha  \left(x z^2-x^3\right),
\end{split}
\end{align}
(see \url{http://agt2.cie.uma.es/form(5).pdf} for a proof). Thus, since $\a\ne 0$ and $p_i\in I$, we have $xz^2-x^3\in I$ and, given that $I$ is a $*$-ideal, we also have $z x^2-z^3\in I$. On the other hand, 
$x p_1=\l_3 x^3+2z x^2+\l_0 z^2x-zx=(\l_0+\l_3) x^3+2z x^2-zx\in I$. But also 
$x p_2=\lambda _0 x^3+2 x^2 z-x^2+\lambda _3 x z^2=(\l_0+\l_3)x^3+2x^2z-x^2\in I$. 
We conclude that $x^2-z x\in I$ and consequently $z^2-z x\in I$. Then, $x^2-z^2\in I$ and, taking into account $p_1$ and $p_2$, we get that $z-x\in I$. The system \eqref{lor} gives another set of elements of $I$ 
\begin{equation}\label{lor2}
    \begin{split}
       q_1= (\l_0+\lambda _3+2) x^2-x, & \ q_3=\lambda _0 t^2+2 t y-t+\lambda _3 y^2-\alpha  x,\\
       q_2=t x+\left(\lambda _0+\lambda _3+1\right) x y, & \  q_4=\lambda _3 t^2+2 t y-\alpha  x+\lambda _0 y^2-y.
    \end{split}
\end{equation}
Since $x\equiv z$ we have 
$$\small\begin{cases}
q_2=t x+\left(\lambda _0+\lambda _3+1\right) x y,\\
q_2^* \equiv x y+\left(\lambda _0+\lambda _3+1\right) tx, 
\end{cases}
$$
which can be written 
$$(tx,xy)\begin{pmatrix}1 & \l_0+\l_3+1\\ \l_0+\l_3+1 & 1\end{pmatrix}\in I\times I.$$
So, we have the cases

 \begin{center}\fbox{Case $\l_2 = \l_1 = 1$, $(\l_0+\l_3+1)^2\ne 1$}\end{center} $(\l_0+\l_3+1)^2\ne 1$. Then, $x t, x y\in I$ and $I\ni 
    x q_3=\lambda _0 x t^2+2 x t y-x t+\lambda _3 x y^2-\alpha  x^2\equiv -\a x^2$. Since $\a\ne 0$, we get that $x^2\in I$ and, from the fact that $q_1\in I$,
    we obtain that $x\in I$. Thus, the universal associative representation $\rho$ is not faithful. 
\begin{center}\fbox{Case $\l_2 = \l_1 = 1$ $\l_0+\l_3+1=1$}\end{center} 
    $\l_0+\l_3+1=1$, that is $\l_0+\l_3=0$. In this case, the polynomials $q_1,\ldots,q_4$ turn out to be 
    \begin{equation}\label{lor3}
        \begin{split}
            q_1=  2 x^2-x, &  \ q_3= \lambda _0 (t^2-y^2)+2 t y-t-\alpha  x,\\
            q_2= t x+x y, & \ q_4= -\lambda _0 (t^2-y^2)+2 t y-\alpha  x-y.
        \end{split}
    \end{equation}   

We may assume that the characteristic of $K$ is other than $2$ because in case $\car(K)=2$, the polynomial $q_1$ is just $x\in I$, and so the universal associative representation is not faithful. So, we assume that $\car(K)\ne 2$.

Let $J\subset I$ be the ideal generated by $q_1,\ldots,q_4$. We claim that
$$J=(x,q_2,q_3,q_4)\cap (2x-1,q_2,q_3,q_4).$$ To prove this, recall first that if $\mathfrak{a},\mathfrak{b}\triangleleft R$ are  coprime ideals of a commutative ring $R$, then $\mathfrak{a}\cap\mathfrak{b}=\mathfrak{a}\mathfrak{b}$.
Therefore, 
$(x,q_2,q_3,q_4)+(2x-1,q_2,q_3,q_4)=(1)$, which implies that 
$(x,q_2,q_3,q_4)\cap (2x-1,q_2,q_3,q_4)=(x,q_2,q_3,q_4) (2x-1,q_2,q_3,q_4)$ and furthermore
\begin{align}
\begin{split}
(x,q_2,q_3,q_4)(2x-1,q_2,q_3,q_4)\subset J\subset(x,q_2,q_3,q_4)\cap (2x-1,q_2,q_3,q_4)=\\ (x,q_2,q_3,q_4)(2x-1,q_2,q_3,q_4).
\end{split}
 \end{align}
Next, we prove that $x\in (2x-1,q_2,q_3,q_4)=:M$ from which we conclude that $x\in J\subset I$ and that the universal associative representation is not faithful. In order to check that $x\in M$, notice that $M\ni 2q_2=2x(t+y)=(2x-1)(t+y)+(t+y)$,
whence $t+y\in M$. Using $q_3$ and $q_4$ we get that $2ty-t-\a x\in M$ and $2ty-y-\a x\in M$. Therefore,
$M\ni 2ty-t-\a x-(2ty-y-\a x)=y-t$ and, since $t+y\in M$, we get that
$y,t\in M$. Then, since $M\ni q_3$ we obtain that $\a x\in M$ and so $x\in M$.
Given that $J=(x,q_2,q_3,q_4)\cap M$ we get $x\in J\subset I$.
\begin{center}\fbox{Case $\l_2 = \l_1 = 1$, $\l_0+\l_3+1=-1$}\end{center}  
$\l_0+\l_3+1=-1$. In this case, the polynomial $I\ni q_1=-x$ and the universal representation
is not faithful.

We summarize our findings above in the theorem below.

\begin{theorem}\label{paella_brasilera_quiero}
The algebra $A_{3,\a}$ does not have a faithful, associative and commutative representation. 
\end{theorem}

\subsection{The $\boldsymbol{\text{A}_{4,\alpha}}$ algebra} 
Recall that the multiplication table for this algebra is $e_1^2=\alpha e_2$, $e^2_2=e_1+e_2$.

Consider the product induced by  $p(a,b)=\l_0 ab+\l_1 ab^*+\l_2a^*b+\l_3 a^*b^*$, with $\l_i\in K$, and the universal associative and commutative representation given by $\U_p=K[x,y,x^*,y^*]/I$. In this case,
the following polynomials generate the ideal $I$.

$\begin{matrix*}[l]
     q_1 = p(x,x)-\alpha y  & =\l_0x^2+(\l_1+\l_2)xx^*+\l_3(x^*)^2-\alpha y \\ 
    q_2= p(y,y)-x-y &=  \l_0y^2 + (\l_1+\l_2)yy^* + \l_3 (y^*)^2 -x-y \\
    q_3 = p(x,y) & =  \l_0xy + \l_1xy^*+\l_2x^*y + \l_3x^*y^*\\ 
    q_4 = p(y,x) &=  \l_0xy + \l_2 xy^*+\l_1x^*y+\l_3x^*y^* \\
    q_1^* = p(x^*,x^*)-\alpha y^* &=  \l_3x^2 + (\l_1+\l_2)xx^*+\l_0(x^*)^2 -\alpha y^* \\
    q_2^* = p(y^*,y^*)-x^*-y^* &=  \l_3y^2 + (\l_1+\l_2)yy^* + \l_0(y^*)^2 - x^*-y^* \\
    q_3 ^* = p(x^*,y^*) &= \l_3xy+\l_2xy^*+\l_1x^*y +\l_0 x^*y^* \\
    q_4^* = p(y^*,x^*) &=  \l_3 xy + \l_1 xy^* + \l_2 x^*y + \l_0 x^* y^*
\end{matrix*}$

If we consider $\l_0 = \l_3$, then we have that $q_1-q_1^* = -\alpha (y-y^*) \in I$ and $y \equiv y^*$ because $\alpha \neq 0$. In addition, if we consider $q_2-q_2^* = -(x-x^*)-(y-y^*) \equiv -(x-x^*) \in I$, we obtain that $x \equiv x^*$. By Proposition~\ref{rose}, the representation can not e faithfu, because $A_{4,\a}$ is not associative.  

We assume in the sequel that $\l_0\ne\l_3$. In this case, we have $q_3-q_4^* = (\l_0-\l_3)(xy-x^*y^*) \in I$ and $xy \equiv x^*y^*$. Then, we can consider the following family of polynomials in $I$

$\begin{matrix*}[l]
& xy-x^*y^* \\
     q_1 = p(x,x)-\alpha y = & \l_0x^2+(\l_1+\l_2)xx^*+\l_3(x^*)^2-\alpha y \\ 
    q_2= p(y,y)-x-y =&  \l_0y^2 + (\l_1+\l_2)yy^* + \l_3 (y^*)^2 -x-y \\
    q_3 = p(x,y) =&   (\l_0+\l_3)xy\ + \l_1xy^*+\l_2x^*y \\ 
    q_4 = p(y,x) =&  (\l_0+\l_3)xy + \l_2 xy^*+\l_1x^*y \\
    q_1^* = p(x^*,x^*)-\alpha y^* =&  \l_3x^2 + (\l_1+\l_2)xx^*+\l_0(x^*)^2 -\alpha y^* \\
    q_2^* = p(y^*,y^*)-x^*-y^*= & \l_3y^2 + (\l_1+\l_2)yy^* + \l_0(y^*)^2 - x^*-y^* \\
    {q_3^*= p(x^*,y^*)=}&\left(\lambda _0+\lambda _3\right) x^* y^*+\lambda _1 x^* y+\lambda _2 x y^*\\
    {q_4^*=p(y^*,x^*)=}&\left(\lambda _0+\lambda _3\right) x^* y^*+\lambda _2 x^* y+\lambda _1 x y^*
\end{matrix*}$\smallskip

Notice that $q_3-q_4 = (\l_1-\l_2)(xy^*-x^*y) \in I$. In case $\lambda_1 \neq \lambda_2$, we have $xy^* \equiv x^*y$.
\begin{center}\fbox{Case $\l_0=0$, $\l_1\ne\l_2$}\end{center}
We denote by $z:=x^*$ and $t:=y^*$. We can take $\l_3=1$ and we have these elements in $I$:

\begin{align*}
   &  q_0 = xy-zt, & \  q_4 =  xy + \l_2 xt+\l_1 zy,  \\
    & q_1 =(\l_1+\l_2)xz+z^2-\alpha y, &   \ q_5 =  x^2 + (\l_1+\l_2)xz -\alpha t, &  \\
   & q_2=  (\l_1+\l_2)yt +  t^2 -x-y, & \ q_6 =  y^2 + (\l_1+\l_2)yt  - z-t. &  \\
    & q_3 =  xy\ + \l_1xt+\l_2 zy, & &
\end{align*}

Now, there are polynomials $s_0,\ldots,s_6$ such that $\sum_0^6 q_i s_i=t$

(see \url{http://agt2.cie.uma.es/form1.pdf} for a proof). Thus, the representation $\rho$ is not faithful.\smallskip

\begin{center}\fbox{Case $\l_0=0$, $\l_1=\l_2$}\end{center}
Again taking $\l_3=1$, we have these elements in $I$:
\begin{align*}
    & q_0 = xy-z t, & \  q_4 = x^2 + 2\l_1 xz -\alpha t, &  \\ 
    & q_1 = 2\l_1 xz+z^2-\alpha y, & \   q_5 = y^2 + 2\l_1 yt  - z-t, &  \\
    & q_2=  2\l_1 yt +  t^2 -x-y, & \  q_6= z t +\lambda _1( z y+ x t). & \\ 
   & q_3 = xy\ + \l_1(xt+zy), &  & & 
\end{align*}

If $\text{char}(K)=2$, then the set of polynomials above is 
$\{xy-zt, z^2-\a y, t^2-x-y, xy+\l_1(x t+z y), x^2-\a t,y^2-z-t\}$. Eliminating (in this order) $x$, $y$, and $t$, we get the polynomials
$$\begin{matrix*}[l]
z^{10}+\alpha ^3 z^5+\alpha ^4 z^4+\alpha ^3 z^4+\alpha ^5 z^2, \\
z^{16}+\alpha ^8 z^4+\alpha ^7 z^4+\alpha ^6 z^4+\alpha ^9 z, \\
\lambda _1 z^{12}+\alpha  z^{10}+\alpha ^2 \lambda _1 z^9+\alpha ^4 \lambda _1 z^6+\alpha ^3 \lambda _1 z^6+\alpha ^5 z^4+\alpha ^4 z^4+\alpha ^6 \lambda _1 z^3
\end{matrix*}$$
whose $\gcd$ is $z$. Consequently, $z\in I$ and the representation is not faithful.

If $\text{char}(K)\ne 2$ then we distinguish two cases
\begin{center}\fbox{Case $\l_0 = 0, \ \l_1 = \l_2 \ne 0$ and $\l_3 \ne 0$}\end{center}
Again there are polynomials $s_0,\ldots,s_5$ such that $\sum_{i = 0}^5 q_is_i = t$ (see \url{http://agt2.cie.uma.es/form2.pdf} for a proof).

This means that, in this case the representation is not faithful.

\begin{center}\fbox{Case $\l_0 = 0, \ \l_1 = \l_2 = 0$ and $\l_3 \ne 0$}\end{center}
Now, we get these elements in $I$:

\begin{equation*}
    \begin{split}
        p_0=  xy-z t,  & \ \ p_3 =   x^2  -\alpha t, \\
        p_1 = z^2+ t^2 -x -(\alpha + 1) y,  & \  \ p_4 =  y^2   - z-t, \\ 
       p_2 =   xy, & \ \ p_5=  z t. 
    \end{split}
\end{equation*}
where $p_1=q_2+q_1$.
Once again, there exist polynomials
$s_0,... ,s_5$ such that $\sum_{i= 0}^5 p_is_i = t$, which means that the representation is not faithful (see \url{http://agt2.cie.uma.es/form3.pdf} for a proof).

\begin{center}\fbox{Case $\l_0 \ne 0$, $\l_1\ne \l_2$ and $\l_0+\l_1+\l_2\ne 0$}\end{center}
From now on, we suppose that $\lambda_0 \neq 0$. Under this assumption $\l_1 \neq \l_2$, then $xy^* \equiv x^*y$.
As a consequence, $(\l_0+\l_3)xy+(\l_1+\l_2) xy^*\in I$. 
Now, we use again the notations $z=x^*$ and $t=y^*$. Again, there exist polynomials $s_0,... ,s_8$ such that $\sum_{i= 0}^8 p_is_i = t$, which means that the representation is not faithful (see \url{http://agt2.cie.uma.es/form(9).pdf} for a proof).

\begin{center}\fbox{Case $\l_0 \ne 0$, $\l_1\ne \l_2$ and $\l_0+\l_1+\l_2= 0$}\end{center}

Substituting $\l_0 = -(\l_1+\l_2)$  in the corresponding polynomials,  we get

$\begin{matrix*}[l]
q_0 = & xy-zt \\
     q_1 =& -(\l_1+\l_2)x^2+(\l_1+\l_2)xz+\l_3z^2-\alpha y \\ 
    q_2 =&  -(\l_1+\l_2)y^2 + (\l_1+\l_2)yt + \l_3 t^2 -x-y \\
    q_3 =&   (-\l_1-\l_2+\l_3)xy\ + \l_1xt+\l_2zy \\ 
    q_4 =&  (-\l_1-\l_2+\l_3)xy + \l_2 xt+\l_1zy \\
    q_5 =&  \l_3x^2 + (\l_1+\l_2)xz-(\l_1+\l_2)z^2 -\alpha t \\
    q_6 = & \l_3y^2 + (\l_1+\l_2)yt  -(\l_1+\l_2)t^2 - z-t \\
    {q_7=}&\left(-\l_1-\l_2+\lambda _3\right) z t+\lambda _1 z y+\lambda _2 x t\\
    {q_8=}&\left(-\l_1-\l_2+\lambda _3\right) z t+\lambda _2 z y+\lambda _1 x t
\end{matrix*}$\smallskip

Again, there exist polynomials $s_0,\ldots, s_8$ such that $\sum_{i= 0}^8 q_is_i = t$, which means that the representation is not faithful (see \url{http://agt2.cie.uma.es/form4.pdf} for a proof).

\begin{center}\fbox{Case $\l_0 \ne 0$, $\l_1 =  \l_2$}\end{center}
If $\l_0 \neq 0$, we can consider $\l_0 = 1$ and we can obtain new polynomials in $I$Althoughht the parameters $\lambda_i$ and $\alpha$ will change, we will denote them in the same way. Then, we have the following polynomials in $I$. 

\begin{equation}\label{jacare}
    \begin{split}
        q_0 = xy-zt, & \  q_4 =  \l_3x^2 + 2\l_1xz+z^2 -\alpha t, \\ 
        q_1 = x^2+2\l_1xz+\l_3 z^2-\alpha y, & \  q_5 = \l_3y^2 + 2\l_1yt + t^2 - z-t  \\ 
        q_2=  y^2 + 2\l_1yt + \l_3 t^2 -x-y, & \  q_6= \left(1+\lambda _3\right) z t+\lambda _1 (z y+ x t). \\ 
        q_3 =  (1+\l_3)xy\ + \l_1(xt+zy), & 
    \end{split}
\end{equation}
To study the above system, we first assume that:

\begin{center}\fbox{$\car(K)\neq 2$}
\end{center}

In this case, we can consider the following subcases.

\begin{center}\fbox{Case $\l_0 = 1$, $\l_1 =  \l_2$ and $\lambda_1 \notin \{-1,0\}$}\end{center}

Again, there exist polynomials $s_0,\ldots, s_6$ such that $\sum_{i=0}^6 q_i s_i = t$ , which means that the representation is not faithful (see \url{http://agt2.cie.uma.es/form5.pdf} for a proof).

\begin{center}\fbox{Case $\lambda_0 = 1$, $\lambda_1 =  \lambda_2$ and $\lambda_1 = 0$}\end{center}
In this case we have the following polynomials in the ideal $I$: 

\begin{equation*}
    \begin{split}
        q_0 = xy-zt, & \ q_4 =  \lambda_3x^2 +z^2 -\alpha t \\ 
        q_1 = x^2+\lambda_3z^2-\alpha y, &  \ q_5 = \lambda_3y^2  + t^2 - z-t,  \\
        q_2 = y^2 + \lambda_3 t^2 -x-y, & \ q_6=\left(1+\lambda _3\right) z t.\\ 
        q_3 = (1+\lambda_3)xy, & 
    \end{split}
\end{equation*}
Again, there exist polynomials $s_0,\ldots, s_6$
such that $\sum_{i=0}^6 q_i s_i =t$ and so the representation is not faithful (see \url{http://agt2.cie.uma.es/form6.pdf} for a proof).
\begin{center}\fbox{Case $\lambda_0 = 1$, $\lambda_1 =  \lambda_2$ and $\lambda_1 = -1$}\end{center}

We have the following polynomials in the ideal $I$: 

\begin{equation*}
\begin{split}
    q_0= xy-zt, & \ q_4 = \lambda_3x^2 -2xz+z^2 -\alpha t,  \\
    q_1 = x^2-2xz+\lambda_3z^2-\alpha y, & \ q_5 = \lambda_3y^2 - 2yt + t^2 - z-t,\\
    q_2= y^2 -2yt + \lambda_3 t^2 -x-y, & \ q_6=\left(1+\lambda _3\right) z t- (z y+ x t). \\ 
    q_3 =  (1+\lambda_3)xy\ -(xt+zy), & \ 
\end{split}
\end{equation*}

then there exist polynomials $s_0,\ldots, s_6$ verifying  $\sum_{i = 0}^6 q_i s_i = t$  and once again the representation is not faithful (see \url{http://agt2.cie.uma.es/form7.pdf}).

To finish, we need to study the system of polynomials (\ref{jacare}) in the case
\begin{center}\fbox{$\car(K) = 2$}
\end{center}

In this setting, the system becomes: 
\begin{equation*}
    \begin{split}
        q_0 =  xy-zt, & \  q_3 =   (1+\lambda_3)xy\ + \lambda_1(xt+zy),  \\ 
        q_1 = x^2+\lambda_3 z^2+\alpha y,& \ q_4 =\lambda_3x^2 +z^2+ \alpha t, \\ 
        q_2=   y^2  + \lambda_3 t^2 +x +y, & \ q_5 =\lambda_3y^2  + t^2 + z+t. \\ 
    \end{split}
\end{equation*}
Again, there exist polynomials $s_0,\ldots,s_5$ such that $\sum_{i = 0}^5 q_i s_i = \alpha t$ and once again the representation is not faithful (see \url{http://agt2.cie.uma.es/form8.pdf}).  

We summarize our findings above in the theorem below.

\begin{theorem}\label{alleap}
The algebra $A_{4,\a}$ does not have a faithful, associative and commutative representation.
\end{theorem}

\subsection{The  $\boldsymbol{\text{A}_{5,\alpha, \b}}$ algebra} 
In this section we study the evolution algebra $A_{5,\a,\b}$, which has a natural basis $\{e_1,e_2\}$ with product defined by $e_1^2=e_1+\a e_2$, $e_2^2=\b e_1+e_2$, where it is assumed that
$\a\b\ne 1$, $\a\ne \b$ and $\a,\b\ne 0$. Consider the product induced by $q(a,b)=l ab+m ab^*+n a^*b+p a^*b^*$, with $l,m,n,p\in K$, and the universal associative and commutative representation given by $\U_q=K[x,y,x^*,y^*]/I$. In this case,
$I=(q(x,x)-x-\a y, q(y,y)-\b x-y,q(x,y),q(y,x))_*$. Thus, $I$ is the ideal generated by the polynomials
\begin{equation}\label{fsys}
    \begin{split}
(1) \ &  l xy+m xt+nz y+pzt, \\
(2) \ &  l x y+m y z+n t x+p t z, \\
(3) \ &  l t z+m y z+n t x+p x y, \\
(4) \  & l t z+m t x+n y z+p x y, \\
    \end{split}
    \ \ 
    \begin{split}
  (5) \ & l x^2+m x z+n x z+p z^2-x-\a y,\\
(6) \  & l z^2+m x z+n x z+p x^2-z-\a t, \\
(7) \ & l y^2+m t y+n t y+p t^2-\b x-y, \\
(8) \  & l t^2+m t y+n t y+p y^2-\b z-t, \\
    \end{split}
\end{equation}
where $z=x^*$ and $t=y^*$.

\begin{center}\fbox{Case $p = l$}
\end{center}
 If this happens then, subtracting (5) from (6), we get that $x-z+\a(y-t)\equiv 0$ and, subtracting (7) from (8), we obtain that $\b(x-z)+(y-t)\equiv 0$. The system that we arrive at has an invertible matrix and hence $x\equiv z$ and $y\equiv t$, which implies, by
Proposition~\ref{rose}, that $\rho$ is not faithful (given that $A_{5,\a,\b}$ is not associative). 

\begin{center}\fbox{Case $p \neq l$}
\end{center}
Next, we have the following subcases:
 
\begin{center}\fbox{Case $p \neq l$, $m \neq n$}
\end{center}
Subtraction between the first two polynomials gives 
$(m-n)(xt-zy)\equiv 0$, which implies that $xt\equiv zy$. Moreover,  subtracting the third polynomial from the first one gives 
$(l-p)(xy-zt)+(m-n)(tx-yz)\equiv 0$, and hence $xy\equiv zt$. 
Thus, the following 
polynomials generate the ideal $I$:
\begin{align*}
L_1= &x y-t z, \ & L_6=&p x^2+m z x+n z x+l z^2-\a t-z, \\
L_2=&t x-y z, \ & L_7=&p t^2+m y t+n y t+l y^2-\b x-y, \\
L_3=&m t x+l y x+p t z+n y z, \ & L_8=&l t^2+m y t+n y t-t+p y^2-\b z, \\
L_4=&n t x+l y x+p t z+m y z, \ & L_9=&p xy+n x t+m y z+l t z,\\
L_5=&l x^2+m z x+n z x-x+p z^2-\a y, \ & L_{10}=& p xy+m x t+n yz+l tz.
\end{align*}

It is interesting to point out that $I$ is invariant under the following automorphism:
$\s\colon K[x,y,z,t]\to K[x,y,z,t]$ with $\{x,y,z,t\}\to\{x,y,z,t\}$, such 
that $x\leftrightarrow z$ and $y\leftrightarrow t$. In fact, we have: $\s(L_i)=-L_i$ for $i=1,2$, $\s(L_3)=L_9, \ \s(L_4)=L_{10}$, $\s(L_5)=L_6, \s(L_7)=L_8$. Hence, $\s(I)\subset I$ and, since $\s$ is an involution, $\s(I)=I$.
\begin{proposition} \label{nonameth} 
In the above settings, we have that $x^2\equiv z^2$ and $y^2\equiv t^2$.
\end{proposition}
\begin{proof}
To begin, we construct some new polynomials in $I$, namely: 
\begin{align*}
 a_1:= & t L_1-y L_2=y^2 z-t^2 z,\\
a_2 := & p a_1 - z L_8=\b z^2-l t^2 z-m t y z-n t y z-p t^2 z+t z,\\
a_3:= &l L_1 -L_3=l (x y-t z)-(l x y+m t x+n y z+p t z)\equiv\\
&-l t z-m y z-n y z-p t z\in I,\\
a_4:= &t a_3+z L_8\equiv -\b z^2-p t^2 z+p y^2 z-t z.  
\end{align*}

Given that $a_1\in I$, one has that $a_4\equiv -\beta z^2-t z\in I$. So, $I\ni\s(a_4)\equiv -\beta x^2-y x$. Consequently $-\beta(z^2-x^2)-t z+x y\equiv 0$ and $L_1\in I$ implies 
$z^2-x^2\equiv 0$.
Next, we prove that we also have $t^2-y^2\equiv 0$. To do this, notice that $L_5$ and $L_6$ give: $(p+l) z^2+(m+n) z x-x-\a y\equiv 0$, $(p+l) z^2+(m+n) z x-z-\a t\equiv 0$ and whence $-x+z-\a y+\a t\equiv 0$, or $x\equiv z+\a(t-y)$. Replacing this value of $x$ in $L_1$, we obtain that 
\begin{equation}\label{ceo} 0\equiv x y-t z\equiv
zy+\a(t-y)y-tz.\end{equation}
Replacing the same value of $x$ in $L_2$ gives:
\begin{equation}\label{oec}
0\equiv tx-yz\equiv t(z+\a(t-y))-yz=tz+\a t(t-y)-yz.
\end{equation}
Then, summing the results of \eqref{ceo} and \eqref{oec}, we have that
$ y(t-y)+t(t-y)\equiv 0$, that is $t^2\equiv y^2$.
\end{proof}
Summarizing, in case $m\ne n$, Proposition~\ref{nonameth} says that $x^2\equiv (x^*)^2$ and $y^2\equiv (y^*)^2$. Then, Lemma \ref{eqp} 
implies that the universal associative and commutative representation can not be faithful (given that $A_{5,\a,\b}$ is not an associative algebra). \begin{center}\fbox{Case $p \neq l$, $m = n = 0$}
\end{center}
In this case, the following polynomials are in $I$:
\begin{align*}
    (1) &\ l xy+pzt, \ & (4) & \ l z^2+p x^2-z-\a t \\
    (2) & \ l t z+p x y \ & (5) & \  l y^2+p t^2-\b x-y\\
    (3) & \ l x^2+p z^2-x-\a y & \ (6) & \ l t^2+p y^2-\b z-t
\end{align*}
The first two polynomials give that $\small(xy,zt)\begin{pmatrix}
    l & p\\
    p & l\end{pmatrix}\in I\times  I$
 \begin{center}\fbox{Case $p\neq l$, $m = n = 0$, $p^2 \neq l^2$}
\end{center}
Then $xy,zt\in I$. Hence, from the third item above, we obtain that 
   $pz^2y-\a y^2\in I$ and so $y^2 t\in I$ and $y t^2\in I$.
   Moreover, the sixth polynomial above gives $l t^3-t^2\in I$ (and $l y^3-y^2\in I$).
   The fourth polynomial gives $px^2t-\a t^2\in I,$ and the fifth gives $pt^2x-\b x^2\in I$. Whence, $x^2 z\in I$ (and $x z^2\in I$). Item (4) gives $lz^2y-zy-\a ty\in I$. Then
   \begin{align}
       \begin{split}
            l(pz^2y-\a y^2)-p(l z^2 y-z y-\a t y)=-l\a y^2+p zy+p\a ty & \in I\\
            -l\a y^3+p z y^2\equiv -\a y^2+pzy^2 & \in I.
       \end{split}
   \end{align}

    Taking into account that $pz^2y-\a y^2\in I$, we also get
  $p(zy^2-z^2y)\in I$.                \begin{center}\fbox{Case $p = m = n = 0$, $l \neq 0$}
\end{center}
 Now, we immediately get $x^2=z^2=y^2=t^2\in I$ and hence (3) and (5) above give $(x,y)\begin{pmatrix}1 & \b\\ \a & 1\end{pmatrix}\in I\times I$. Since the matrix above is invertible, we get that $x,y\in I$.
\begin{center}\fbox{Case $p \notin \{l,0\}$, $m =n = 0$, $p^2 \neq l^2$}
\end{center}
Then $zy^2-z^2 y$, $xt^2-x^2t\in I$. Also, we had proved 
$px^2t-\a t^2\in I$ and $pt^2x-\b x^2\in I$. Thus, $\a t^2\equiv \b x^2$ and $\a y^2\equiv \b z^2$. We introduce these congruences in (3)-(6) above and obtain that all the following polynomials are in $I$
\begin{align*}
    lx^2+p \frac{\a}{\b}y^2-x-\a y, & \ l y^2+p \frac{\b}{\a}x^2-\b x-y, \\
    l\frac{\a}{\b}y^2+px^2-z-\a t, & \ l \frac{\b}{\a}x^2+p y^2-\b z-t.
\end{align*}
From the first of these equations, we obtain that $l x^3-x^2\in I$ and 
$ \frac{p}{\b}y^3- y^2\in I$. But we had that $l y^3-y^2\in I$ and
hence $\tiny (y^3,y^2)\begin{pmatrix}p & l\\ -\b & -1\end{pmatrix}\in I\times I$.
\begin{center}
    \fbox{Case $p \notin \{l,0\}$, $m = n = 0$, $p^2 \neq l^2$, $p \neq \beta l$}
\end{center}
 We have $y^2\in I$ and $t^2\in I$ (and given that $\a t^2\equiv\b x^2$ we also have $x^2\in I$ and $z^2\in I$). We conclude as before that $x,y,z,t\in I$. 
 \begin{center}\fbox{Case $p\notin \{l,0\}$, $m =n = 0$, $p^2 \neq l^2$, $p = \beta l$}
\end{center}
Then, we have the next polynomials in $I$
\begin{equation}\label{cold}
    \begin{split}
        l x^2+\alpha  l y^2-x-\alpha  y &, \\
        \beta ^2 l x^2+\alpha  l y^2-\alpha  \beta  t-\beta  z,&  
    \end{split}
    \
    \begin{split}
    \beta ^2 l x^2+\alpha  l y^2-\alpha  \beta  x-\alpha  y,& \\
     \beta  l x^2+\alpha  \beta  l y^2-\alpha  t-\alpha  \beta  z.& 
    \end{split}
\end{equation}
so that from the first and third polynomials we get $(1-\a\b)x\equiv l(1-\b^2)x^2$. Since $\a\b\ne 1$, we have that $x$ is a scalar multiple of $x^2$ modulo $I$. But we had that $zx^2\equiv 0$, so we also get that $xz\equiv 0$. Now, multiplying the polynomials in \eqref{cold} by $x$, reducing, and taking into account that $xz, xy, zt\equiv 0$, we get that 
$lx^3-x^2$, $\b l x^3-\a t x$, and $\b l x^3-\a x^2$ are in $I$. Therefore, 
$(\b-\a)x^2\in I$. Since $\a\ne\b$ we obtain that $x^2\in I$. Hence, $z^2\in I$
and the relation $\a t^2\equiv \b x^2$ implies that $t^2,y^2\in I$. Then, $x+\a y, \a\b x+\a y\in I$ and we conclude that 
$x,y\in I$.
\begin{center}\fbox{Case $p\ne l$, $m = n = 0$, $p^2=l^2$,  $p = -l$}
\end{center}
Under this assumption, we have that the following polynomials are in $I$:
\begin{align*}
    (1) \ & xy-zt, \  & (4) \ &  l(y^2- t^2)-\b x-y, & \\
        (2) \ & l (x^2- z^2)-x -\a y ,\  & (5) \ & l(t^2- y^2)-\b z-t, & \\
       (3) \ &  l (z^2- x^2)-z-\a t, \ & &
\end{align*}
Adding the second and third polynomials, we obtain $x+\a y+z+\a t\equiv 0$. Similarly, from the fourth and fifth polynomials, we get that
These two congruences, we get that $x\equiv -z$ and
$y\equiv -t$ from these two congruences. This gives $x^2\equiv z^2$ and $y^2\equiv t^2$, which implies, by Lemma~\ref{eqp}, that the associative and commutative representation is not faithful. 
\begin{center}\fbox{Case $p\neq l$, $m = n = 1$}
\end{center}
Recall that we are under the hypothesis that $l\ne p$. We subtract the first and third polynomials in the fundamental system \eqref{fsys} to get that $xy-zt\in I$.
So, $I$ is the ideal generated by $l_i$ $(i=1,\ldots,8)$, where
\begin{align*}
    l_1= &x y-t z, \ & l_5=&p t^2+ 2y t+l y^2-\b x-y, \\
    l_2=& l x y+ t x+ y z+p t z, \ & l_6=&l t^2+ 2 y t+p y^2-\b z-t,\\ 
    l_3=&l x^2+ 2 x z +p z^2-x-\a y, \ & l_7=&l t z+t x+y z+p xy. \\
    l_4=&p x^2+ 2 x z +l z^2-z-\a t,& & 
\end{align*}
Notice that we have  $\a yt-\b x z = \sum_{i=1}^7 l_is_i \in I$ (see \url{http://agt2.cie.uma.es/form11.pdf}). 

Moreover,

we also have that $\a y^2-\b z^2\in I$ (see \url{http://agt2.cie.uma.es/form12.pdf}) and $\a t^2-\b x^2\in I$.
Furthermore, for a new collection of $s_i$ we have  $-\b t x+\b x^2+\b y z-\b z^2= \sum_{i=1}^7 l_is_i \in I$ (see \url{http://agt2.cie.uma.es/form13.pdf})

so that $- t x+ x^2+ y z- z^2\in I$. Finally, $ -\a t x+\a y z+\b t x-\b y z \in I$ (see \url{http://agt2.cie.uma.es/form14.pdf})

from which we deduce that $\a(yz-xt)-\b (yz-xt)\equiv 0$, that is, $yz-xt\equiv 0$ (since $\a\ne\b$). 

Now, $I\ni (xy-zt)t-(xt-yz)y=y^2z-t^2z$ and $t^2x-y^2x\in I$. Also 
$$\begin{cases}zl_5 \equiv (p+l)t^2z+2ytz-\b xz-y z\\ zl_6 \equiv (l+p)t^2z+2ytz-\b z^2-t z,\end{cases}$$
and, subtracting these polynomials, we have that $I\ni \b(xz-z^2)+yz-tz$. So, applying the involution we have that both
$$\begin{cases} \b(xz-z^2)+yz-tz, \\ \b(xz-x^2)+tx-yx
\end{cases}$$
belong to $I$. Subtracting once more we get that $\b(x^2-z^2)+tx-yx-yz+tz\in I$, and using that $xy\equiv tz$ and $xt\equiv yz$ we get that $x^2-z^2\in I$. Furthermore, since we had $\a t^2\equiv \b x^2$ and $\a y^2\equiv \b z^2$, we conclude
that $y^2\equiv t^2$. Then, Lemma~\ref{eqp} implies that the universal associative and commutative representation is not faithful.

\begin{theorem}\label{Ipanema_girl}
The evolution algebra $A_{5,\a,\b}$, where $\alpha,\beta\in K^\times$, $\alpha\ne\beta$ and $\alpha\beta\ne 1$, has no associative and commutative faithful representation. 
\end{theorem}

\subsection{The $\boldsymbol{\text{A}_{5,\alpha, \a}}$  algebra} 

We analyze the evolution algebra with natural basis $\{e_1,e_2\}$ such that $e_i^2=e_i+\a e_j$ with $i\ne j$ and $\a\ne 0,\pm 1$.  We have proved that the Hopf algebra representing the affine group scheme $\affaut(A_{5,\a,\a})$ is $\H_{5,\a,\a}= K[x]/(x^2-x)$, which is two-dimensional and isomorphic as $K$-algebra to $K^2$ with component-wise product.
As was proved in \cite{squares}, we have a faithful associative and commutative representation of $A_{5,\a,\a}$ in the algebra
$K^2$ with exchange involution and componentwise product relative to $p(a,b)=ab+\a a^*b^*$. So, in this case, the Hopf algebra representing
$\affaut(A_{5,\a,\a})$ supports a faithful associative and commutative representation. 
We have $\U_p=K[x,y,z,t]/I$ (where $z$ plays the roll of $x^*$ and $t$ the one of $y^*$), and $I$ is generated by 
$$\begin{cases}
i_0=x^2-x+\a z^2  -\a y, & i_3=t^2-t+\a y^2  -\a z,\\
i_1=\alpha  x^2+z^2-z-\a t, & i_4=x y+\a t z,\\
i_2= \alpha  t^2+y^2-y-\a x, & i_5=z t+\a x y.\\
\end{cases}$$
Note that from $i_4$ and $i_5$ we get $xy, zt\in I$.
So from $i_0$ we have $z^2 y-y^2\in I$ and whence $x^2 t-t^2\in I$. From $i_3$ we 
deduce that $y^2 z-z^2\in I$ and hence $t^2 x-x^2\in I$. Now,
$$ z^2 y-y^2 \in I \Rightarrow z^2y t-y^2 t \in I\Rightarrow y^2t\in I\Leftrightarrow t^2y\in I.$$ Similarly, from $ y^2 z-z^2 \in I$ we get that $x y^2 z -z^2x \in I$ and so $z^2x\in I$ and $x^2 z\in I$. Also, $i_3$ gives $t^3-t^2\in I$ and so
$y^3-y^2\in I$. Moreover, $i_0$ gives $x^3-x^2\in I$ and so $z^3-z^2\in I$. Furthermore, since $z^2y-y^2\in I$, we have that $z^2y^2-y^3\in I$ and, since $y^2z-z^2\in I$, we also
have that $y^2z^2-z^3\in I$. Then, $y^3-z^3\in I$ which implies that $y^2-z^2\in I$ and $t^2-x^2\in I$. Now, 
$i_0-i_3=x^2-t^2+\a(z^2-y^2)-x-\a y+t+\a z\in I$ implying
$$\begin{cases}-x-\a y+t+\a z\in I\\ -z-\a t +y+\a x\in I\end{cases}\qquad
\begin{cases}(t-x)+\a(z-y)\in I\\ \a(t-x)+(z-y)\in I\end{cases}$$ and whence
$t-x,z-y\in I$, given the conditions on $\a$. Now, using $t\equiv x$ and $z\equiv y$ modulo $I$ we find that $i_0\equiv x^2-x+\a y^2-\a y$
and $i_1=\a x^2+y^2-y-\a x$ (while $i_2, i_3$ collapse). So, $I=(x^2+\a y^2-x-\a y, \a x^2+y^2-y-\a x, xy)$. Thus, 
$$(x^2,y^2)\begin{pmatrix}1 & \a\\ \a & 1\end{pmatrix}=(x,y)\begin{pmatrix}1 & \a\\ \a & 1\end{pmatrix}$$ which implies $x^2\equiv x$, $y^2\equiv y$ and so 
$\U_p=K[x,y,z,t]/(x^2-x,y^2-y, xy, t-x,z-y)\cong K[x,y]/(x^2-x,y^2-y,xy)=K\bar 1\oplus K\bar x\oplus K\bar y$ is $3$-dimensional with basis 
$\{\bar 1,\bar x,\bar y\}$. Of course $\T_p=K\bar x\oplus K\bar y$ with multiplication $\bar x^2=\bar x$, $\bar y^2=\bar y$, $\bar x\bar y=0$, that is, 
$\T_p\cong K^2$ with component-wise product. Thus, $\T_p\cong \H_{5,\a,\a}$. An easy computation reveals that $\aut(A_{5,\a,\a})=\aut^*(\T_p)=\aut(\T_p)$. Again, the tight universal $q$-algebra agrees with the underlying algebra of the Hopf algebra representing its affine group scheme of automorphisms.
\begin{theorem}
    For the $A_{5,\a,\a}$ algebra ($\a\ne 0,\pm 1$) the tight $p$-algebra $\T_p$
    agrees with $\H_{5,\a,\a}$ (as algebra) and is isomorphic to $K^2$ with component-wise product. For the automorphism groups we have $\aut(A_{5,\a,\a})=\aut^*(\T_p)=\aut(\T_p)$.
\end{theorem}

\subsection{The $\boldsymbol{\text{A}_{5}}$ algebra}
In this case a natural basis is $\{e_1,e_2\}$, with 
$e_1^2=e_1-e_2$ and $e_2^2=-e_1+e_2$. Assume that $\car(K)\ne 2$. In \cite{squares}, we found a faithful associative and commutative representation $\rho\colon A_5\to K^2$ 
such that $e_1\mapsto u_1:=(-\frac{1}{4},1)$ and $e_2\mapsto u_2:=(-\frac{1}{4},-1)$,
where the algebra structure of $K^2$ is $(x,y)(z,t)=(xz-yt,xt+yz)$, its involution is $(x,y)^*=(x,-y)$ and we consider the product $p(a,b)=-2 ab+2a^*b^*$. 
If $\car(K)=2$, there is also a faithful commutative and associative representation.
Consider $K^2$ with componentwise product and exchange involution. Take $p(a,b)=ab+a^*b^*$ and $\rho\colon A_5\to K^2$ such that $e_1\mapsto (1,0)=:u_1$ while  $e_2\mapsto (0,1)=:u_2$. Then
$\rho(e_1^2)=\rho(e_1+e_2)=(1,1)$ and 
$p(u_1,u_1)=u_1^2+u_2^2=u_1+u_2=(1,1)$. Similarly
$\rho(e_2^2)=\rho(e_1+e_2)=(1,1)$ and 
$p(u_2,u_2)=u_2^2+u_1^2=u_1+u_2=(1,1)$. Also 
$p(u_1,u_2)=u_1u_2+u_2u_1=0$. 
\begin{theorem} Regardless of the characteristic of the ground field, the $A_5$ algebra has a faithful associative and commutative representation. If $\car{(K)}\ne 2$ then the product considered is $p(a,b)=-2ab+2a^*b^*$, while in case $\car{(K)}=2$ the product is $p(a,b)=ab+a^*b^*$. 
\end{theorem}
\subsection{The $\boldsymbol{\text{A}_{6}}$  algebra}
For this algebra, a natural basis is $\{e_1,e_2\}$ with $e_1^2 = 0$, $e_2^2 = e_1$. Since this algebra is associative and commutative,  a faithful representation of this kind is $K\times K$ with product $(x,y)(z,t)=(yt,0)$, identity involution, and $p(a,b)=ab$. The representation is $\rho\colon A_6\to K\times K$ with $e_1\mapsto(1,0)$ and $e_2\mapsto(0,1)$.
\subsection{The $\boldsymbol{\text{A}_{7}}$  algebra}

For this algebra, a natural basis $\{e_1,e_2\}$ is such that $e_1^2 = e_1$ and $e_2^2 = 0$. This algebra is also associative and commutative, and a faithful representation of this kind is $K\times K$ with product $(x,y)(z,t)=(xz,0)$, identity involution, and $p(a,b)=ab$. The representation is $\rho\colon A_7\to K\times K$ with $e_1\mapsto(1,0)$ and $e_2\mapsto(0,1)$. 
\textcolor{magenta}{}
\subsection{The  $\boldsymbol{\text{A}_{8,\alpha}}$  algebra}
For this algebra, a natural basis $\{e_1,e_2\}$ is such that $e_1^2 = e_1$ and $e_2^2 = \a e_1$ with $\a \in K^\times$.
If $\car(K)\ne 2$ consider $$q(a,b)=k_\a(ab+a^*b^*)+h_\a(ab^*+a^*b),$$
where $k_\a=\frac{1}{4} (1-16 \alpha )$ and 
$h_\a=\frac{1}{4} (1+16 \alpha )$. 
In this case, a  faithful associative representation (relative to $q$ above) is $K\times K$  with (associative) product $(x,y)(z,t)=(xz-yt, xt+yz)$ and involution $(x,y)^*=(x,-y)$.
The faithful representation is $\psi\colon A_{8,\a}\to K\times K$, with 
$\psi(e_1)=(1,0)$, $\psi(e_2)=(0,1/4)$. The nonassociative product $q$ above, particularized to this algebra, gives $(x,y)\cdot (z,t)=(16\alpha ty+xz,0)$.

Next, we analyze the $\car(K)=2$ case.
The set of generators of $I$ is 
\begin{equation}\label{orpoc}
    \begin{split}
        l x^2+m x z+n x z+p z^2-x, &\cr
        l z^2+m x z+n x z+p x^2-z, &\cr
        l y^2+m t y+n t y+p t^2-\alpha  x, & \cr
        l t^2+m t y+n t y+p y^2-\alpha  z, \cr
    \end{split}
    \ \ 
    \begin{split}
        l x y+m t x+n y z+p t z, &\cr
        l x y+m y z+n t x+p t z,& \cr
        l t z+m y z+n t x+p x y, & \cr
        l t z+m t x+n y z+p x y.& 
    \end{split}
\end{equation}
\centerline{\fbox{Case $l=p$}}
From the first two equations, we obtain that $x\equiv z$. So, putting $\lambda=l+m+n+p$, the following elements are in $I$: $\l x^2-x, 
l (y^2+t^2)+(m+n) t y-\alpha  x, \ l (x y+t x)+m t x+n x y, \ l (x y+t x)+m x y +n t x.$
Then $(m+n)(t x+x y)\equiv 0$ and we divide the proof into two subcases.

 \centerline{\fbox{Case $l = p$,  $m =   n$}} 
 Then $\l=0$ and so $x\in I$. Hence, the universal associative and commutative representation is not faithful.

\centerline{\fbox{Case $l = p$,  $m \ne n$}}  
Then $t x+x y\equiv 0$ and the following elements are in $I$:  $\l x^2-x, \
l (y^2+t^2)+\l t y-\alpha  x, \
\l x y$. Since $\l=m+n\ne 0$, then we have that $x y\equiv t x\equiv 0$. Multiplying the second equation by $x$, we get that $x^2\equiv 0$, so $x\in I$ and the universal associative and commutative representations are not faithful.

\centerline{\fbox{Case $l\ne p$}}
In this case, we have that $(l+p)(x^2+z^2)+x+z\equiv 0$, which implies that 
$(x+z)[(l+p)(x+z)+1]\equiv 0$. 
Define $k:=(l+p)^{-1}\ne 0$ and consider the ideals  $J_1:=(x+z)+I$ and $J_2:=(x+z+k)+I$, which are coprime ideals
so that $J_1J_2\subset I\subset J_1\cap J_2=J_1J_2$. Therefore, $I=J_1J_2$.
To prove that $x\in I$ we will prove that $x\in J_1$ and $x\in J_2$. We start proving that $x\in J_1$, dividing the proof into a few subcases.

     \centerline{\fbox{Case $l\ne p$ and $m=n$}} 
     Since $x+z\in J_1$, we can replace $z$ with $x$ in the system \eqref{orpoc} to get that $J_1$ contains the elements 
    \begin{equation}\label{lahemos}
        \begin{split}
           (1) \ (l+p) x^2+x,& \  (4) \ l y^2+p t^2+\alpha  x, \cr
            (2) \ (l+n) x y+(n+p) t x, & \ (5) \ l t^2+p y^2+\alpha  x.\cr 
         (3) \ (l+n) t x+(n +p)x y, & 
        \end{split}
    \end{equation}
   Then, from the second and third elements above, we get that the polynomial $\left((n+p)^2+(l+n)^2\right) tx\in J_1$ and so $(p+l)^2 t x\in J_1$, whence
   $t x\in J_1$. So, $(l+n)xy$ and  $(n+p)xy\in J_1$ and therefore $(l+p)xy\in J_1$, which implies that $xy\in J_1$. Thus, by multiplying the last element in \eqref{lahemos} by $x$, we obtain that $x^2\in J_1$ and hence, using the first element in (\ref{lahemos}), we obtain that $x\in J_1$.
   
    \centerline{\fbox{Case $l\ne p$, $m\ne n$, $p=0$}}  In this case, from \eqref{orpoc}, we obtain the following set of elements in $J_1$ 
   \begin{align*}
      (1) \ &\l x^2+x,\   & (5)\ (l+m)x y+n t x,\\
       (2) \ & l y^2+(m+n)t y+\a x, \  & (6)\ m x y+(l+n)t x,\\
      (3) \ & l t^2+(m+n)t y+\a x, \   & (7) \ n x y+(m+l)t x, \\
    (4) \ &(l+n) xy+m t x,&  \
   \end{align*}
   where $\l=l+m+n\ne 0$ (if $\l=0$, then $x\in J_1$). Then, the last two elements imply that
   $xy$ and $tx\in J_1$, since the determinant of the matrix $\tiny\begin{pmatrix}m & l+n\\n & l+m\end{pmatrix}$ is nonzero. Therefore, from the second element, we obtain that $\a x^2\in J_1$; hence, using the first element, we obtain that $x\in J_1$.
   
    \centerline{\fbox{Case $l\ne p$,  $m\ne  n$, $p\ne 0$}}  Now, using \eqref{orpoc} a set of elements of $J_1$  is 
   \begin{align*}
      (1) \ & \l x^2+x, \ & (5) & \ (n+p)xy+(l+m)t x,\cr
      (2) \ & (l+m) x y+(n+p) t x, \ & (6) &\ l y^2+(m+n) t y+p t^2-\alpha  x,\cr
      (3) \ &  (l+n) x y+(m+p) t x, \ & (7) &\ l t^2+(m+n) t y+p y^2-\alpha  x,\cr
       (4) \ & (m+p)xy+(l+n)t x, \ & &
   \end{align*}
   where $\l=l+m+n+p$. Again $\l\ne 0$ (on the contrary $x\in J_1$). Using the second and third elements, we obtain, as before, that $xy, tx\in J_1$. From the last element, we obtain that $\a x^2\in J_1$ and, using the first element, we conclude that $x\in J_1$.
To summarize, we have that $x\in J_1$. We now prove that $x\in J_2$. 

Since $x+z+k\in J_2$ we have $z\equiv x+k$ modulo $J_2$. 
we proceed assuming $(l+p)k=1$ and $x+z+k\in J_2$. 
If we apply the congruence $x+z+k\in J_2$ to \eqref{orpoc}, and take into account that 
$(l+p)k=1$, we get the following elements in $J_2$:
\begin{equation}\label{olluruz}
    \begin{split}
       (1) \ \l k x+\l x^2+k^2 p,\cr
        (2) \ (l+m)  x y+ (n+p) t x+ k m y+ k p t,\cr
   (3) \  (l+n)  x y+(m+p) t x+ k n y+ k p t,\cr
     (4) \ (m+p) x y + (l+n)  t x+  k l t+ k m y,\cr
    \end{split}
    \ \
    \begin{split}
    (5) \ (n+p) x y+ (l+m)  t x+ k l t+ k n y,\cr
    (6) \ l y^2+(m+n) t y+p t^2+\alpha  x,\cr
   (7) \  k \alpha +l t^2 + (m+n) t y +p y^2 +\a x,
    \end{split}
\end{equation}
 where $\l=l+m+n+p$. The last two elements give that $(l+p)(y^2+t^2)\equiv k\a$ or $y^2+t^2\equiv k^2\a$ (modulo $J_2$). Moreover, the elements no. 2 and 4 in \eqref{olluruz} can be written as below
 $$\begin{pmatrix}l+m & n+p\\ m+p & l+n\end{pmatrix}\begin{pmatrix} xy\cr tx\end{pmatrix}\equiv\begin{pmatrix}km & kp\\km &kl \end{pmatrix}\begin{pmatrix}y\\t\end{pmatrix}\ \hbox{ (mod }\ J_2).$$
 Notice that the determinant of the matrix on the left-hand side above is $\l(l+p)$. We now divide the proof in two cases.
 
  \centerline{\fbox{Case $l\ne p$, $\lambda \neq 0$}}  Then 
 $$\begin{pmatrix} xy\cr tx\end{pmatrix}\equiv\begin{pmatrix}l+m & n+p\\ m+p & l+n\end{pmatrix}^{-1}\begin{pmatrix}km & kp\\km &kl \end{pmatrix}\begin{pmatrix}y\\t\end{pmatrix} \hbox{ (mod }\ J_2) $$
 and hence $xy \equiv \frac{k (m y+n t)}{\lambda }$ and $t x \equiv \frac{k (\lambda +n)}{\lambda } t+\frac{ k m}{\lambda } y$ (so $xy +\frac{k (m y+n t)}{\lambda } \in J_2$ and $t x + \frac{k (\lambda +n)}{\lambda } t+\frac{ k m}{\lambda } y \in J_2$).
  Replacing these values of $xy$ and $tx$ in the system \eqref{olluruz} we get 
  \begin{equation}\label{olluru}
      \begin{split}
    \l k x+\l x^2+k^2 p,\cr 
    (m+n) (t+y),
      \end{split}
      \ \
      \begin{split}
    l y^2+(m+n) t y+p t^2+\alpha  x,\cr
    k \alpha +l t^2 + (m+n) t y +p y^2 +\a x.
      \end{split}
  \end{equation}
If $m+n\ne 0$, then $y+t\in J_2$, so $y^2+t^2\in J_2$ and, since $y^2+t^2\equiv k^2\a$, we get a nonzero scalar $0\ne k^2\a\in J_2$. So, $J_2=(1)$. If $m+n=0$, we have the following elements in $J_2$:
\begin{equation}\label{ollur}
    \begin{split}
         \l k x+\l x^2+k^2 p,\cr 
    l y^2+p t^2+\alpha  x,\cr
    k \alpha +l t^2 + p y^2 +\a x, 
    \end{split}
    \ \
    \begin{split}
         \l xy + k n (y+ t),\cr 
     \l t x + k (\lambda +n) t+k n y.
    \end{split}
\end{equation}
But then 
\begin{align}
\begin{split}
K^*\ni \a=\alpha  \lambda  \left(\l k x+\l x^2+k^2 p\right)+\left(k \lambda  p+\lambda ^3 x^2\right) \left(l y^2+p t^2+\alpha  x\right)+\cr 
\left(k \lambda  p+\lambda +\lambda ^3 x^2\right) (k \alpha +l t^2 + p y^2 +\a x)+\cr 
\left(\lambda ^2 t+\lambda ^3 t x+\lambda ^3 x y\right) (\l xy + k n (y+ t))+\cr 
\left(\lambda ^2 t+\lambda ^3 t x+\lambda ^3 x y\right)(\l t x + k (\lambda +n) t+k n y)\in J_2.
\end{split}
\end{align}
So, $J_2=(1)$.

\centerline{\fbox{Case $l\ne p$, $\lambda = 0$}}   Then, the first element in \eqref{olluruz} implies that $J_2=(1)$, unless $p=0$. So, we investigate the case $\l=p=0$. We have that $l+m+n=0$ and hence the following polynomials are in $J_2$:
\begin{equation}\label{nojagac}
    \begin{split}
        (1) \ (l+m)  (x y+  t x)+ k m y,\cr
   (2) \ m( x y+ t x)+ k n y,\cr
   (3) \  m (x y +  t x)+  k l t+ k m y,
    \end{split}
    \ \ 
    \begin{split}
    (4) \ (l+m) (x y+  t x)+ k l t+ k n y,\cr
    (5) \ l (y^2+ t y)+\alpha  x,\cr
   (6) \ k \alpha +l (t^2 + t y)  +\a x. 
    \end{split}
\end{equation}
 Remembering that $l \neq p = 0$ and adding the second and third elements above, we obtain that
$kn y+k l t+k m y\in J_2$. Equivalently,
$k l (y+t)\in J_2$ so $y+t\in J_2$. But then, using the fifth element, we obtain that $x\in J_2$.

Thus, we have proved that $x\in J_2$ regardless of the case considered. Since we have already proved that $x\in J_1$, we conclude that $x\in I=J_1\cap J_2$ and the
universal associative and commutative representation is not faithful.
We have proved that, for $A_{8,\a}$,  the existence of a faithful associative and commutative representation depends on the characteristic of the field. More precisely, we proved the following theorem.

\begin{theorem}
The algebra $A_{8,\a}$ has a faithful associative and commutative representation if and only if $\car(K)\ne 2$. 
\end{theorem}

As a consequence of the results proved in Section~\ref{morrocotudo}, we obtain the following characterization of the existence of a faithful associative and commutative representation of a $2-$dimensional evolution algebra.
\begin{theorem}
Let $A$ be a $2$-dimensional evolution algebra over a field $K$. Then, $A$ has a faithful associative and commutative representation if and only if $\aut(A)\ne\{1\}$.
\end{theorem}

Compiling the information on the Hopf algebra representing the affine group scheme $\affaut(A)$ in the above theorem we obtain  Table~\ref{table1}  (up to scalar extension):

\medskip 

\begin{table}[ht]
\begin{center}
\setcellgapes{2pt}
\makegapedcells
\settowidth\rotheadsize{Perfect}
\settowidth\rotheadsize{Non-Perfect}
\begin{tabular}{|C{0.4cm}|c|c|c|c|c|}
\hline 
 &  $A$ & $\aut(A) $ &   $\H$: Hopf alg. & Univ. repr. & $\T_p$ \\
  &     &         &  of $\affaut(A)$ & &\\
          \hline
                   
  & $A_1$ & $\mu_2(K)$ & $K^2$ & Faithful & $\cong\H$\\
   \cdashline{2-6}
   
  & $A_{2,\a}$ & $\mu_3(K)\rtimes\mu_2(K)$ & $KC_3\o KC_2$ & Faithful & $\cong\H$\\
   \cdashline{2-6}
   
  & $A_{3,\a}$ &  $1$  & $K$  & Not faithful & $-$\\
  \cdashline{2-6}
  
  & $A_{4,\a}$ &  $1$  & $K$  & Not faithful&$-$\\
  \cdashline{2-6}
  
  & $A_{5,\a,\b}$, \tiny ($\a\ne\b$) &  $1$  & $K$  & Not faithful &$-$\\
  \cdashline{2-6}
  
  & $A_{5,\a,\a}$ & $\m_2(K)$ & $K^2$ & Faithful & $\cong\H$\\
   
   \hline 
   \multirow[t]{-5.5}{*}{\rothead{Perfect}}
  & $A_5,\ \tiny\car=2$ & $(K,+)$ & $K[x]$ & Faithful &{\color{blue}}\\ 
  \cdashline{2-6}
  
  & $A_5,\ \tiny\car\ne 2$ & $(K,\cdot)$ & $K[x^\pm]$ & Faithful &{\color{blue}}\\
  \cdashline{2-6}
  
  & $A_6$ & $K^\times\times K$ & $K[x,y^{\pm}]$ & Faithful &\\
  \cdashline{2-6}
  
  & $A_7$ & $K^\times$ & $K[x^\pm]$ & Faithful &\\
  \cdashline{2-6}
  
  & $A_{8,\a},\ \tiny\car\ne 2$ & $\mu_2(K)$ & $K^2$ & Faithful &\\
  \cdashline{2-6}
   \multirow[t]{-4}{*}{\rothead{Non-Perfect}}
  
  & $A_{8,\a},\ \tiny\car= 2$ & $1$ & $K(\epsilon)=\frac{K[x]}{(x^2)}$ & Not faithful & $-$\\
  \cdashline{2-6}

\hline 

\end{tabular}
\caption{\small Hopf algebras and universal representations}
\label{table1}
\end{center}
\end{table}

Note that if $A$ is a finite-dimensional algebra such that the representing Hopf algebra $\H$ of $\affaut(A)$ is the ground field $K$, then $\aut(A)=1$, but in some cases, we may have $\aut(A)=1$ without having $\H=K$. Such is the case for $A_{8,\a}$ in characteristic $2$.
\begin{remark}\label{despensa}\rm 
Let $A$ be a finite-dimensional algebra over an algebraically closed field $K$ and $\H$
the representing Hopf algebra of $\affaut(A)$, which we know to be finitely generated. It is well known that for a finite-dimensional perfect evolution algebra $A$, its group of automorphisms is finite (\cite[Theorem 4.8]{Elduque2}). So $\hom(\H,K)$ is a finite group and this implies that $\H$ is a finite-dimensional $K$-algebra (this is rather standard:  taking $\H=K[x_1,\ldots,x_n]/I$ with $V(I)=\{a_1,\ldots,a_m\}$ finite and 
$a_i=(a_{i1},\ldots,a_{in})$, consider $f_1=\prod_i (x_1-a_{i1}),\ldots, f_n=\prod_i (x_n-a_{in})$. For each $j$
we have $f_j\in I(V(I))=\{g\in K[x_1,\ldots,x_n]\colon g(V(I))=0\}=\sqrt{I}$ by the Hilbert's Nullstellensatz. 
Then some power of $f_j$ is in $I$ (for any $j$), and since
each $f_j$ depends only on the variable $x_j$, we conclude that $K[x_1,\ldots,x_n]/I$ is finite-dimensional. 
If the ground fields are not algebraically closed, we can extend scalars to the algebraic closure to get the same result.
\end{remark}

Assume that $A$ is a finite-dimensional evolution algebra over an algebraically closed field $K$ with $\aut(A)=1$. Denote by $\H$ the Hopf algebra representing $\affaut(A)$. We have a dichotomy,  $\H=K$ or $\H\ne K$.

In the second case, we have that
$\hom_{\alg_K}(\H,K)$ has only one element: the augmentation $\epsilon\colon\H\to K$, because  $\affaut(A)(K)=\aut(A)=1$.
  For any maximal ideal $\mi\triangleleft\H$, we have $\H=\mi\oplus K 1$ and a homomorphism $\H\to K$, just the projection. Since there is only one such homomorphism, there is only one maximal ideal, so $\H$ is a local algebra. 
  
We know $\H = K[x_1,x_2,\ldots,x_n]/I$ with $n\ge 1$ and $I$ an ideal. A maximal ideal $\mi$ in $\H$ comes from a maximal ideal $\mathfrak{M}$ in $K[x_1,x_2,\ldots,x_n]$ such that $I \subseteq \mathfrak{M}$. Since we have taken $K$ to be an algebraically closed field, we have that $\mathfrak{M} = (x_1-a_1,x_2-a_2,\ldots,x_n-a_n)$, with $a_i \in K$ for $i = 1,\ldots,n$. But there is a standard automorphism $K[x_1,\ldots,x_n] \to K[x_1,\ldots,x_n]$ such that $x_i \mapsto x_i+a_i$ for every $i  = 1,\ldots,n$. Hence, we can consider $\mathfrak{M} = (x_1,x_2,\ldots, x_n)$ up to isomorphism. In addition, the algebra $\H$ is finite-dimensional (see Remark \ref{despensa}), and so it is artinian, which implies that ${\rm rad}(\H) = \mi $ is nilpotent. Then \cite[11.4, {Theorem}]{Waterhouse} implies that $\car(K)=p$, a prime. Since $\H$ is in particular finitely generated, then $\affaut(A)$ is a finite group scheme in the terminology of \cite[2.4]{Waterhouse}. Furthermore, this group scheme is connected (see \cite[6.6]{Waterhouse}) because $\H$ is local and therefore there are no idempotents in $\H$ others than $0$ and $1$ (see \cite[5.5, Corollary]{Waterhouse}). So, we can apply \cite[14.4]{Waterhouse} to get that $\H\cong K[x_1,\ldots,x_n]/(x_1^{p^{e_1}},\ldots, x_n^{p^{e_n}})$. Conversely, if we have $\H$ as before, then it is easy to realize that $\hom(\H,K)$ has cardinal $1$. 
So we claim:
\begin{proposition}\label{singleton}
Let $A$ be a finitely generated algebra in $\alg_K$ (with $K$ algebraically closed). The
affine scheme $h^A$ has only one rational point, that is, 
$\hom_{\alg_K}(A,K)$ is a singleton if and only if $\H=K$, or $\car(K)=p$ prime and $A\cong K[x_1,\ldots,x_n]/(x_1^{p^{e_1}},\ldots, x_n^{p^{e_n}})$, for some exponents $e_i$. 
\end{proposition}

As a consequence of the information summarized in Table \ref{table1}, we see that 
\begin{theorem}
For a $2$-dimensional evolution algebra $A$ over $K$, we have $\aut(A)=\{1\}$ if and only if there are no faithful $p$-algebra for $A$.
If $\H$ is the representing Hopf algebra of $\affaut(A)$, then there is a faithful associative and commutative representation for $A$ if and only if $\text{char}(K)\neq2$ and  $\H\not\cong K$, or $\text{char}(K)=2$,  $\H\not\cong K(\epsilon)$ and $\H\not\cong K$. Furthermore, if $A$ is perfect and has a faithful tight $p$-algebra, then $\T_p\cong \H$ (as algebras)
\end{theorem}

\section{Consequences for higher dimensional algebras}\label{banco}

Following \cite{shestakov}, we introduce a minor variation  on the notion of associative representation in \cite{shestakov} to produce the concept of associative and commutative representation of an evolution algebra. Fix $A$ a finite-dimensional evolution $K$-algebra, with $\dim(A)=n$. Take a polynomial $p$ in the free associative algebra with involution $\hbox{Ass}\langle a,b,a^*,b^*\rangle$ such that $p$ is a $K$-linear combination 
of $\{ab,ab^*,a^*b,a^*b^*\}$. Consider the polynomial algebra $$R:=K[x_1,\ldots,x_n,x_1^*,\ldots,x_n^*]$$ with the involution such that $x_i\mapsto x_i^*$.
If $B=\{e_i\}_{i=1}^n$ is a natural basis of $A$ with $e_i^2=\sum_j \w_{ji}  e_j$, $\w_{ji}\in K$ (the structure constants relative to $B$), then we can define  the
$K$-linear map $\rho\colon A\to\U_p$ such that 
$\rho(e_i)=x_i$ and $\U_p:=R/I$ where $I$ is the $*$-ideal of $R$ $*$-generated by the elements 
$p(x_i,x_j)$ with $i\ne j$ and $p(x_i,x_i)-\sum_j \w_{ji} x_j$. We will call $\U_p$ the universal associative and commutative representation of $A$ relative to the specified $p$ and $B$.
Thus, $\rho$ is a linear map $A\to\U_p$ where $\U_p$ is an (associative and commutative) algebra with involution, and $\rho$ satisfies 
$\rho(ab)=p(\rho(a),\rho(b))$ for any $a,b\in A$.
Furthermore, if 
$\tau\colon A\to X$ is any other linear map to an associative and commutative algebra with involution $(X,*)$, satisfying $\tau(ab)=p(\tau(a),\tau(b))$, then there is a unique $*$-homomorphism of (associative) algebras $F\colon\U_p\to X$ such that $F\rho=\tau$. Any linear map $\tau\colon A\to X$, where $(X,*)$ is an associative, commutative algebra with involution satisfying $\tau(ab)=p(\tau(a),\tau(b))$ is called an associative and commutative representation of $A$ for the product given by $p$.

\medskip

For $B$ an associative algebra with involution $*$, we will denote by $\hbox{\rm Sym}(B,*)$ the set of elements $b\in B$ such that $b^*=b$.
\begin{proposition}\label{ventilador}
    Let $A$ be a finite-dimensional perfect evolution algebra and $\rho\colon A\to\U_p$ its universal associative and commutative representation relative to $p=\l_0 ab+\l_1 ab^*+\l_2 a^*b+\l_3 a^*b^*$. If $\rho(A)\subset\hbox{\rm Sym}(\U_p,*)$, then the kernel $\ker(\rho)$ is an ideal such that $A/\ker(\rho)$ is associative. In particular, if $\rho$ is faithful, $A$ is an associative algebra and $\aut(A)\ne 1$.  
\end{proposition}
\begin{proof}
Let $\{e_i\}_1^n$ be the natural basis chosen for the computation of $\U_p$. 
We know that $\U_p=R/I$, with $R=K[x_1,\ldots,x_n,x_1^*,\ldots,x_n^*]$ and
 $\rho(e_i)=\bar x_i$ for any $i$. 
The ideal $I$ is $*$-generated by the polynomials $p(x_i,x_i)-\sum_j \omega_{ji} x_j$ 
and $p(x_i,x_j)$ ($i\ne j$). Define $\l=\sum_0^3\l_j$. 
\begin{enumerate}
\item If $\l=0$ then $0=\sum_j\omega_{ji} \bar x_j$, which implies $\bar x_j=0$ for any $j$, given that the matrix $(\omega_{ji})$ is invertible. In this case, $\rho=0$, and we are done.
\item If $\l\ne 0$, we have $\rho(xy)=\l\rho(x)\rho(y)$ for any $x,y\in A$. Then, there is a vector space isomorphism $\hat\rho\colon A/\ker(\rho)\to\rho(A)$ given by $\hat\rho(\bar x)=\rho(x)$.
But $\hat\rho(\bar x\bar y)=\rho(xy)=\l\rho(x)\rho(y)$, whence
$$\hat\rho((\bar x\bar y)\bar z)=\hat\rho(\overline{xy}\bar z)=\l\rho(xy)\rho(z)=\l^2\rho(x)\rho(y)\rho(z)
$$
and similarly $\hat\rho(\bar x(\bar y\bar z))=\l^2\rho(x)\rho(y)\rho(z)$. So $(\bar x\bar y)\bar z-\bar x(\bar y\bar z)\in\ker(\hat\rho)=0$.
\end{enumerate}
In any case, $A/\ker(\rho)$ is associative.  If $\rho$ turns out to
be faithful, then $A$ is associative and perfect, so $A\cong K^n$ for a certain $n$. In this case, its group of automorphisms is not trivial.
\end{proof}
As a corollary of Proposition \ref{ventilador} we obtain the following result.

\begin{corollary}
 If $A$ is a perfect evolution algebra of finite dimension with a faithful representation $\rho$ and $A$ is not associative, then 
$\rho(A)\not\subset\hbox{\rm Sym}(\U_p,*)$.    
\end{corollary}

Let $\sigma\colon E\to (A,*)$ be a representation of an evolution algebra $E$ in the associative algebra with involution $(A,*)$, relative to a product $p=\l_0 ab+\l_1 ab^*+\l_2 a^*b+\l_3 a^*b^*$ in the free associative algebra with involution $K[a,b,a^*,b^*]$ (the elements $\l_i\in K$). Then, for any $f\in\aut_K(E)$, there is another representation $\sigma f\colon E\to A$ relative to the same $p$:
$$\sigma f(xy)=\sigma(f(x)f(y))=p[\sigma(f(x)),\sigma(f(y))].$$
Thus, the set $\rep_p(E,A)$ of all linear maps $\sigma\colon E\to A$ such that $\sigma(xy)=p[\s(x),\s(y)]$ for $x,y\in E$, is an $\aut(E)$-set via the action $\aut(E)\times\rep_p(E,A)\to\rep_p(E,A)$ such that 
$f\cdot\sigma=\sigma f^{-1}$.
It is well known that the orbit $[\sigma]$ of $\sigma$ under the action of $\aut(E)$ is isomorphic as $\aut(E)$-set to $\aut(E)/\aut(E)_\sigma$, where $\aut(E)_\sigma$ is the isotropy group of $\sigma$, that is, the stabilizer of $\sigma$, which is the subgroup $\aut(E)_\sigma:=\{f\in\aut(E)\colon f\cdot\sigma=\sigma\}$.  
If $\sigma$ is faithful, its isotropy group is trivial, so the orbit of $\sigma$ is isomorphic as an $\aut(E)$-set to $\aut(E)$. Then,  $\left \vert[\sigma] \right \vert=\vert\aut(E)\vert$, meaning that there are as many faithful representations as the cardinal of $\aut(E)$. 
Assume now that $\s(E)\subset\hbox{Sym}(A,*)$, that is, $\s(x)^*=\s(x)$ for any $x\in E$. Let $\l=\sum_{i=0}^3\l_i$ (recall that
$p=\l_0 ab+\l_1 ab^*+\l_2 a^*b+\l_3 a^*b^*$).

We know that $\s(xy)=p[\s(x),\s(y)]=\l \s(x)\s(y)$ for any $x,y\in E$.
When $\l=0$ we get 
$\s(E^2)=0$. So, if the
representation $\s$ is faithful we get that $E^2=0$ and hence $E$ is a trivial evolution algebra (a zero-product algebra). If $\l\ne 0$
then for any $x,y,z\in E$ we have 
$$\s((xy)z)=\l\s(xy)\s(z)=\l^2\s(x)\s(y)\s(z)=\l\s(x)\s(yz)=\s(x(yz)).$$
The associator of any three elements is in $\ker(\s)$ so when
$\s$ is faithful we get that $E$ is associative. The discussion above allow us to state the following result.

\begin{proposition}\label{queremoschurros}
If $E$ is an evolution algebra which is not associative and 
$\s\colon E\to (A,*)$ is a faithful (associative and commutative) representation, relative to $p$, then $\s(E)\not\subset\hbox{\rm Sym}(A,*)$. Furthermore, $\tau\colon E\to (A,*)$ defined as 
$\tau(x):=\s(x)^*$ is also a (faithful) representation and 
$\tau\ne\s$.
\end{proposition}

A fact related to the information in Table~\eqref{table1} is the following proposition.

\begin{proposition}
  Let $A$ be an evolution $K$-algebra with faithful universal associative and commutative representation $\rho\colon A\to\U_p$
  and assume that $\rho(A)$ is $*$-closed, that is, $\rho(A)^*=\rho(A)$. Then $\aut(A)\ne 1$  or $A$ is associative. 
\end{proposition}
\begin{proof}
   Under the hypothesis of the proposition, there is a unique linear map $\sharp \colon A\to A$ such that $\rho(x^\sharp)=\rho(x)^*$ for any $x\in A$. Then we
   have $$\rho(x^{\sharp\sharp})=\rho(x^\sharp)^*=\rho(x)^{**}=\rho(x),$$ implying that
   $x^{\sharp\sharp}=x$ for any $x$. Similarly, $(xy)^\sharp=x^\sharp y^\sharp$ for any $x$ and $y$. If $\sharp$ is the identity, then $\rho(x)^*=\rho(x)$ and applying Proposition \ref{queremoschurros} we get that $A$ is associative. 
   
   \end{proof}
\medskip

 We have seen in the precedent paragraphs that in the case of a $2$-dimensional perfect evolution algebra $A$, the $p$-algebra $\T_p$ agrees with the representing Hopf algebra $\H$ of $\affaut(A)$. So it seems interesting to study those evolution algebras, in general, such that 
$\T_p\cong\H$.

Assume that for a perfect, finite-dimensional evolution algebra $E$ we have $\T_p\cong \H$, where $\H$ is the Hopf algebra of $\affaut(E)$. Suppose also that $E$ is not associative. Then, we have two possibilities:
\begin{enumerate}
    \item If $*=1_{\T_p}$, then Proposition \ref{ventilador} implies that $\rho\colon E\to\T_p$ is not faithful but the quotient $E/\ker(\rho)$ is associative. Note that $\ker(\rho)\ne E$, since on the contrary we would have $\T_p=0$, which is impossible as $\T_p=\H\ni 1$.
    \item Otherwise, $*\in\aut(\T_p)$ and $*\ne 1_{\T_p}$, so 
    $\aut(\T_p)\cong\aut(\H)\ne 1$. Consequently, $\H\ne K$ so we have $\aut(E)\ne 1$ or $\H$ is of the form described in Proposition \ref{singleton}. Thus, $\H\cong\T_p$ is a local algebra different from $K$ and the characteristic of $K$ is prime.
\end{enumerate}
Summarizing this paragraph, we have:

\begin{proposition}\label{bodorrio}
    Let $E$ be a finite-dimensional perfect evolution $K$-algebra with $\H=\T_p$. Denote by  $*$ the canonical involution of $\T_p$.  
    \begin{enumerate}
        \item If $*=1_{\T_p}$, then the universal $\rho\colon E\to\T_p$ is not faithful, $\ker(\rho)\ne E$  and the quotient $E/\ker(\rho)$ is an associative algebra. 
        \item If $*\ne 1_{\T_p}$ then $\H\not\cong K$ and hence:
        \begin{enumerate}
            \item Either $\aut(E)\ne 1$, or
            \item $\car(K)$ is prime and $\H$ is of the form described in  Proposition \ref{singleton}.
\end{enumerate}
    \end{enumerate}
 In particular, if $E$ is a simple evolution algebra which is not associative, then either
(i) $\aut(E)\ne 1$ or (ii) $\H$ is of the form described in Proposition~\ref{singleton}, with $K$ of prime characteristic.
So, in characteristic zero, simple nonassociative evolution algebras with $\H=\T_p$ have nontrivial automorphism group.
\end{proposition}

Consequently, an investigation of simple nonassociative evolution algebras in prime characteristic, whose associated Hopf algebra is as in
Proposition~\ref{singleton}, could complete the result in Proposition~\ref{bodorrio}.

Finally, let's make an observation about (possibly infinite-dimensional) evolution algebras: If we consider an evolution algebra $A$ whose graph is of type $E$ in the figure below, it includes a copy of $A_{3,\a}$, which we have previously demonstrated lacks a faithful associative and commutative representation. Therefore, $A$ has no faithful associative and commutative representation either. 
The same argument can be applied to any evolution algebra whose graph is of type $F$. Concerning those with graph of type $G$, it will depend on the scalars $\a$ and $\b$ such that $e_1^2=e_1+\a e_2$, 
$e_2^2=\b e_1+e_2$. If $\a\ne \b$, $\a\b\ne 1$ and $\a,\b\ne 0$, this algebra contains a copy of $A_{5,\a,\b}$, which also lacks a faithful associative and commutative representation (as we proved before).

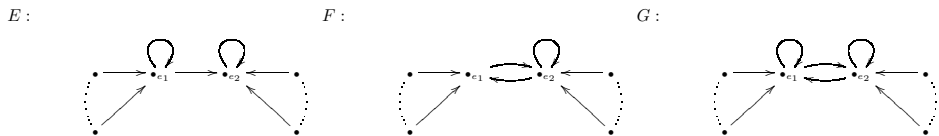
\begin{table}[H]
\scalebox{0.65}{
\begin{tabular}{ccc}
$\tiny E:\xymatrix{
   & {}    &     &  \\
   & {\bullet}  \ar[r]&  {\bullet}_{e_{1}} \ar@(lu,ru)  \ar[r]&   {\bullet}_{e_{2}} \ar@(lu,ru) & {\bullet}  \ar[l]  \ar@{.}@/^/[d] \\
     & {\bullet}  \ar[ur]  \ar@{.}@/^/[u]&       & & {\bullet}  \ar[ul]     \\
           }$
& $F:\tiny \xymatrix{
   & {}    &     &    \\
   & {\bullet}  \ar[r]&  {\bullet}_{e_{1}}   \ar@/^.4pc/[r]&   {\bullet}_{e_{2}} \ar@/^.4pc/[l]\ar@(lu,ru) & {\bullet}  \ar[l]  \ar@{.}@/^/[d] \\
     & {\bullet}  \ar[ur]  \ar@{.}@/^/[u]&       & & {\bullet}  \ar[ul]     \\
           }$ &
            $G:\tiny\xymatrix{
   & {}    &     &    \\
   & {\bullet}  \ar[r]&  {\bullet}_{e_{1}}  \ar@(lu,ru) \ar@/^.4pc/[r]&   {\bullet}_{e_{2}} \ar@/^.4pc/[l]\ar@(lu,ru) & {\bullet}  \ar[l]  \ar@{.}@/^/[d] \\
     & {\bullet}  \ar[ur]  \ar@{.}@/^/[u]&       & & {\bullet}  \ar[ul]     \\
           }$ 

\end{tabular}}

\end{table}

\section*{Acknowledgement}

Part of the research work that led us to this article was carried out during research stays in Brazil by the first, fourth, fifth and sixth authors. These authors would like to thank the Federal University of Santa Catarina for their affability and generosity.
The research leading to this article was conducted, in part, during research stays in Spain by the second and third authors. These authors express their gratitude to the Universidad de Málaga for its warm hospitality and generous support.

\section{Declarations}

\subsection*{Ethical Approval:}

This declaration is not applicable.

\subsection*{Conflicts of interests/Competing interests:} We have no conflicts of interests/competing interests to disclose.

\subsection*{Authors' contributions:}

All authors contributed equally to this work. 

\subsection*{Data Availability Statement:} The data that support the findings of this study are available within the article as well as in the M\'alaga University Server {\bf agt2.cie.uma.es} at \url{http://agt2.cie.uma.es/polynomials.pdf}.

\bibliographystyle{acm}
\bibliography{ref}

\end{document}